\documentclass[12pt]{article}
\usepackage{amsmath, graphicx, amsfonts,amssymb, calrsfs}
\usepackage{amsfonts,mathrsfs, color, amsthm}
\usepackage{bm}
\usepackage{dsfont}
\usepackage{subcaption}
\usepackage{float}

\usepackage{enumitem}
\usepackage{hyperref}

\makeatletter
\newcommand*\bigcdot{\mathpalette\bigcdot@{0.7}}
\newcommand*\bigcdot@[2]{\mathbin{\vcenter{\hbox{\scalebox{#2}{$\m@th#1\bullet$}}}}}
\makeatother

\hypersetup{
colorlinks   = true, 
urlcolor     = blue, 
linkcolor    = blue, 
citecolor   = red
}

\def\sphere{S^{n-1}}

\def\R{\mathbb{R}}

\numberwithin{equation}{section}

\newtheorem{theorem}{Theorem}[section]
\newtheorem{lemma}[theorem]{Lemma}
\newtheorem{remark}[theorem]{Remark}
\newtheorem{proposition}[theorem]{Proposition}
\newtheorem{corollary}[theorem]{Corollary}

\newtheorem{definition}[theorem]{Definition}

\def\bt{\begin{theorem}}
\def\et{\end{theorem}}
\def\be{\begin{equation}}
\def\ee{\end{equation}}
\def\bl{\begin{lemma}}
\def\el{\end{lemma}}
\def\br{\begin{remark}}
\def\er{\end{remark}}
\def\bc{\begin{corollary}}
\def\ec{\end{corollary}}
\def\bd{\begin{definition}}
\def\ed{\end{definition}}
\def\bp{\begin{proposition}}
\def\ep{\end{proposition}}
\def\ball{B^n_2}

\begin{document}
\title{The $m$th order Orlicz projection bodies
}
\author{Xia Zhou, Deping Ye and Zengle Zhang}
\date{}
\maketitle
\begin{abstract} 
Let $M_{n, m}(\mathbb{R})$ be the space of $n\times m$ real matrices. Define $\mathcal{K}_o^{n,m}$ as the set of convex compact subsets in $M_{n,m}(\mathbb{R})$ with nonempty interior containing the origin $o\in M_{n, m}(\mathbb{R})$, and $\mathcal{K}_{(o)}^{n,m}$ as the members of $\mathcal{K}_o^{n,m}$ containing $o$ in their interiors. Let $\Phi: M_{1, m}(\mathbb{R}) \rightarrow [0, \infty)$ be a convex function such that $\Phi(o)=0$ and $\Phi(z)+\Phi(-z)>0$ for $z\neq o.$ In this paper, we propose the $m$th order Orlicz projection operator $\Pi_{\Phi}^m: \mathcal{K}_{(o)}^{n,1}\rightarrow  \mathcal{K}_{(o)}^{n,m}$, and study its fundamental properties, including the continuity and affine invariance. We establish the related higher-order Orlicz-Petty projection inequality, which states that the volume of $\Pi_{\Phi}^{m, *}(K)$, the polar body of $\Pi_{\Phi}^{m}(K)$, is maximized at origin-symmetric ellipsoids among convex bodies with fixed volume. Furthermore, when $\Phi$ is strictly convex, we prove that the maximum is uniquely attained at origin-symmetric ellipsoids.  Our proof is based on the classical Steiner symmetrization and the Fiber symmetrization.  

We also investigate the special case for $\Phi_{Q}=\phi\circ h_Q$, where $h_Q$ denotes the support function of $Q\in \mathcal{K}^{1, m}_o$ and $\phi: [0, \infty)\rightarrow [0, \infty)$ is a convex function such that $\phi(0)=0$ and $\phi$ is strictly increasing on $[0, \infty).$ When $\phi(t)=t^p$ for $p\geq 1$, the $m$th order Orlicz projection operator  $\Pi_{\Phi_Q}^m$ reduces to the special cases in the recent works \cite{Haddad-Ye-2023,Haddad-ye-lp-2025} up to a constant. We  establish a higher-order Orlicz-Petty projection inequality related to $\Pi_{\Phi_Q}^{m, *} (K)$. Although $\Phi_Q$ may not be strictly convex, we are able to characterize the equality under the additional assumption on $\phi$, such as the strict convexity of $\phi$. 

\end{abstract}

\section{Introduction}\label{introduction}
Stemming from the work on the Orlicz projection and centroid bodies by Lutwak, Yang and Zhang \cite{orlicz projection,orlicz centroid}, the Orlicz-Brunn-Minkowski theory of convex bodies (compact convex subsets in the $n$-dimensional Euclidean space $\mathbb{R}^n$ with nonempty interiors) has attracted increasing interest in convex geometry and related fields, such as analysis and partial differential equations. This theory incorporates the Orlicz spaces into the geometric theory of convex bodies and has achieved significant advancements with the introduction of the Orlicz addition of convex bodies by Gardner, Hug and Weil \cite{Gardner-Hug-Weil-2014}, and independently by Xi, Jin and Leng \cite{Xi-Jin-Leng-2014}.  The Orlicz addition generates a non-homogeneous perturbation of convex bodies. When combined with the volume functional, it naturally leads to the Orlicz surface area measures of convex bodies. Characterizing these surface area measures forms the central goal of the Orlicz-Minkowski problem, initiated by Haberl, Lutwak, Yang and Zhang \cite{the even orlicz minkowski problem}. Since its inception, the Orlicz-Brunn-Minkowski theory has undergone rapid development. Some of the key contributions to this area include \cite{o1,o2, Huang-He-2012,o5,o6-Li-2014,o9,o10,o11,o12} for the Orlicz-Minkowski problem, \cite{Hong-Ye-Zhang-2018} for the $p$-capacitary Orlicz-Minkowski problem,  \cite{do 1,GHSD2019,GHXD2020,Huang-Xing-Ye-Zhu-IUMJ,do 3.1,Liu-Lu-2020,Xing-Ye-2020,Zhu-Ye-2018} for the general dual Orlicz-Minkowski problem, \cite{Gardner-Ye-2015,Zhu-Xu-2014} for the dual Orlicz-Brunn-Minkowski theory,  \cite{BoroJDG 2013,lin 2017,lin and xi,orlicz projection,orlicz centroid,Ye-2015,Zhu-2012-star-body} for the Orlicz affine isoperimetric inequalities, \cite{Luo-Ye-Zhu-2020, Xing-Ye-Zhu-2022,  zbc-hh-dpy-2018} for Orlicz-Petty bodies and related polar Orlicz-Minkowski problems, among others.  

Our main concern in this paper is the projection bodies, which play fundamental roles in numerous fields, such as measure theory, convex geometry, geometric tomography, optimization and functional analysis (see, e.g., \cite{pro application 1,pro application 2,Geometric tomography,pro application 0,pro application 4,schneider}). For the sake of consistency in notations, we identify $\mathbb{R}^n$ with the space of $n \times 1$ real matrices $M_{n,1}(\mathbb{R})$.
By $\mathcal{K}_{(o)}^{n,1}$, we mean that the set of convex bodies in $M_{n,1}(\mathbb{R})$ containing the origin $o$ in their interiors. Let $\varphi: \mathbb{R} \rightarrow [0,\infty)$ be a convex function such that $\varphi(0)=0$ and $\varphi(-t)+\varphi(t)>0$ for $t \in \mathbb{R} \setminus \{0\}$. The Orlicz projection body of $K \in \mathcal{K}_{(o)}^{n,1}$, denoted by $\Pi_{\varphi}K$, is a convex body in $M_{n,1}(\mathbb{R})$ defined through its support function \cite{orlicz projection}:
\begin{align}\label{def-orlicz-petty-body}
h_{\Pi_{\varphi}K}(x)=\inf\left\{\frac{1}{t}>0:\int_{S^{n-1}}\varphi\bigg(t\frac{x\bigcdot u}{ h_K(u)}\bigg)h_K(u)dS_K(u)\leq n V_n(K)\right\}\ \ \mathrm{for}\ \ x\in M_{n,1}(\mathbb{R}).
\end{align}
Here $S^{n-1}$ is the unit sphere of $M_{n,1}(\mathbb{R})$, $x \bigcdot u$ means the standard inner product of $x$ and $u$, $S_K$ denotes the surface area measure of $K$, $V_n(K)$ is the volume of $K$, and $h_K$ is the support function of $K$ (see Section \ref{notation} for their definitions). Special cases of the Orlicz projection bodies include the $L_p$ projection body for $p> 1$ by Lutwak, Yang and Zhang \cite{lp proj ineq 1}, and the asymmetric $L_p$ projection body for $p>1$ by Ludwig \cite{ludwig 2005} and Haberl and Schuster \cite{general lp affine iso ine}. 

The affine isoperimetric inequality related to the projection bodies is particularly important in applications. For the Orlicz projection bodies, such an affine isoperimetric inequality, often referred to as the Orlicz-Petty projection inequality \cite{orlicz projection}, states that
\begin{align}\label{orlicz-petty-projection-ine}
\frac{V_n(\Pi_{\varphi}^*K)}{V_n(K)}\leq\frac{V_n(\Pi_{\varphi}^*B_2^n)}{V_n(B_2^n)} \ \ \mathrm{for\ any } \ \ K\in\mathcal{K}_{(o)}^{n,1},
\end{align}
where $B_2^n$ is the unit ball in $M_{n,1}(\mathbb{R})$ and $\Pi_{\varphi}^*K=(\Pi_{\varphi} K)^*$ is the polar body of $\Pi_{\varphi} K$, i.e., $\Pi_{\varphi}^*K=\{y\in M_{n,1}(\mathbb{R}): x\bigcdot y\leq 1, x\in \Pi_{\varphi} K\}$. The Orlicz-Petty projection inequality extends the $L_p$ affine isoperimetric inequalities for the (asymmetric) $L_p$ projection bodies \cite{general lp affine iso ine, lp proj ineq 1, Petty isoperimetric problem}. Important applications and contributions for the various Petty projection bodies include, e.g.,  \cite{BoroJDG 2013, Haberl-2019-complex-affine-isop- inequalities, lin and xi,proj ineq 1,Lutwak 1986,proj inequality 2,pro application 4} for geometric inequalities, \cite{lp affine appl 1.1,asym affine lp sobo ine,an asymmetric affine polva-szego prin,lin 2017,sharpaffine lp sobol ine,Wang-tuo-2012,zhang affine sobolev ine} for functional inequalities, and \cite{lin 2021} for inequality for sets of finite perimeter.
 
 Let $m\in \mathbb{N}$ where $\mathbb{N}$ is the set of positive natural numbers. Inspired by the work of Schneider \cite{Schneider 0} for the higher-order difference body and related Rogers-Shephard inequality, Haddad, Langharst, Putterman, Roysdon and Ye introduced the $(L_p,Q)$-projection bodies for $p\geq 1$ in \cite{Haddad-Ye-2023,Haddad-ye-lp-2025}, and established related  higher-order Petty projection inequalities. In this setting, we consider the space of $n \times m$ real matrices $M_{n,m}(\mathbb{R})$ and denote by $\mathcal{K}_{o}^{1,m}$ the set of convex bodies in $M_{1,m}(\mathbb{R})$ containing $o$.  For $p\ge1$, $Q\in \mathcal{K}_{o}^{1,m}$ and $K\in\mathcal{K}_{(o)}^{n,1}$,  the $(L_p,Q)$-projection body  of $K$, denoted by $\Pi_{Q,p}K$, is defined through its support function \cite{Haddad-ye-lp-2025}:
\begin{align*}
h_{\Pi_{Q,p}K}(\pmb{x})^p=\int_{S^{n-1}}h_Q(v^{\mathrm{T}}{}_{\bigcdot}\, \pmb{x})^p h_K^{1-p}(v)d S_K(v)\ \ \mathrm{for}\ \  \pmb{x}\in M_{n,m}(\mathbb{R}), 
\end{align*} where $v^{\mathrm{T}}$ is the transpose of $v$, and $\pmb{x}=(x_1,\cdots, x_m)\in M_{n,m}(\mathbb{R})$ with each $x_i\in M_{n,1}(\mathbb{R})$ for $i\in \{1,\cdots,m\}$, and  $v^{\mathrm{T}}{}_{\bigcdot}\, \pmb{x}=(v\bigcdot x_1,\cdots,v\bigcdot x_m)$. Let $\left\{e_{1, i}\right\}_{i=1}^m$ be the canonical basis in $M_{1, m}(\mathbb{R})$ and $\operatorname{conv}(E)$ denote the closed convex hull of $E$. When $p=1$ and  $$Q=\Delta_m=\operatorname{conv}\left\{o, -e_{1,1}, \cdots, -e_{1, m}\right\},$$ 
then $$h_{Q}(v^{\mathrm{T}}{}_{\bigcdot}\, \pmb{x})=\max  \big\{0, -x_1\bigcdot v, \cdots, -x_m\bigcdot v \big\},$$ and $\Pi_{Q,p}K$ reduces to the $m$th higher-order projection body $\Pi^m K$ \cite{Haddad-Ye-2023} given by its support function: 
\begin{align*}
h_{\Pi^m K}(\pmb{x})=\int_{S^{n-1}} \max  \big\{0, -x_1\bigcdot v, \cdots, -x_m\bigcdot v\big \} d S_{K}(v).
\end{align*} When $m=1$, $\Pi_{Q,p}K$ covers the $L_p$ projection bodies in \cite{complex projection body,general lp affine iso ine,ludwig 2005,lp proj ineq 1,Petty isoperimetric problem} as its special cases.  The higher-order $L_p$ Petty projection inequality states that 
\begin{align}\label{higher-order-lp-ine}
\frac{V_n(K)}{V_n(B_2^n)}\leq \bigg(\frac{V_{nm}(\Pi^{*}_{Q,p}K)}{V_{nm}(\Pi^{*}_{Q,p}B_2^n)}\bigg)^{\frac{p}{m(p-n)}}.
\end{align} Here $\Pi^{*}_{Q,p}K=(\Pi_{Q,p}K)^*$ is the polar body of $\Pi_{Q,p}K$. See \cite{Haddad-Ye-2023,Haddad-ye-lp-2025} for more details and in particular, the characterization of equality for the above inequality. The higher-order theory has attracted increasing interest, and here we mention a few contributions in the higher-order setting: \cite{J. Haddad,mth order weighted projection body,mth order affine polya szego principle,higher order reverse isoperimetric inequality,higher order lp mean zonoids}.

 Motivated by the Orlicz projection body \cite{orlicz projection} and  the $(L_p,Q)$-projection body \cite{Haddad-Ye-2023,Haddad-ye-lp-2025}, in this paper, we study the $m$th order Orlicz projection body. To this end, let  $\mathcal{C}$ be the set of convex functions $\Phi: M_{1,m}(\mathbb{R})\rightarrow [0,\infty)$ satisfying that $\Phi(o)=0$ and $\Phi(z)+\Phi(-z)>0$ for $z\ne o$. For $\Phi\in\mathcal{C}$, the $m$th order Orlicz projection body of $K\in\mathcal{K}_{(o)}^{n,1}$, denoted as $\Pi_{\Phi}^{m}K$, is defined via its support function
\begin{align*}
h_{\Pi_{\Phi}^{m}K}(\pmb{x})=\inf\left\{\frac{1}{\lambda}>0: \int_{S^{n-1}}\Phi\bigg(\lambda\frac{{u}^{\mathrm{T}}{}_{\bigcdot}\,  {\pmb{x}}}{ h_K({u})}\bigg)h_K({u})dS_K({u})\leq nV_n(K)\right\} 
\end{align*} for $\pmb{x}\!=\!(x_1, \cdots, x_m)\!\in \!M_{n,m}(\mathbb{R})$ with each $x_i\!\in \!M_{n,1}(\R)$ for $i\!\in \!\{1, \cdots, m\}$. Denote by $\mathcal{K}_{(o)}^{n,1}$ the set of convex bodies in $M_{n,1}(\mathbb{R})$ containing the origin $o$ in their interiors. 
We will show that $\Pi_{\Phi}^{m}K \subset M_{n,m}(\mathbb{R})$ is a convex body containing $o$ in its interior, and the operator $\Pi_{\Phi}^{m}: \mathcal{K}_{(o)}^{n,1}\rightarrow \mathcal{K}_{(o)}^{n, m}$ is continuous in terms of the Hausdorff metric and is affine invariant. With the help of these properties, together with the classical Steiner symmetrization (see e.g., \cite{schneider}) and the Fiber symmetrization by McMullen \cite{Mcmullen-1999}, we can establish the following higher-order Orlicz-Petty projection inequality. See Bianchi, Gardner and Gronchi \cite{Bianchi-Gardner-Gronchi-2017} and Ulivelli \cite{Ulivelli} for more details on the Fiber symmetrization and its extensions.    

\vskip 2mm \noindent {\bf Theorem \ref{Main-Theory-1}.} {\em  
Let $\Phi\in\mathcal{C}$  and $\Pi^{m,*}_{\Phi}K=(\Pi^{m}_{\Phi}K)^*$ be the polar body of $\Pi^m_{\Phi}K$. For $K\in \mathcal{K}_{(o)}^{n,1}$, let  
\begin{align*}
\Gamma_{\Phi}(K)=\frac{V_{nm}(\Pi_{\Phi}^{m,*}K)}{V_{n}(K)^m}. 
\end{align*} Then, among $K\in\mathcal{K}_{(o)}^{n,1}$, $\Gamma_{\Phi}(K)$ 
is maximized at origin-symmetric ellipsoids, i.e., \begin{align} \Gamma_{\Phi}(K)\leq \Gamma_{\Phi}(\ball)\ \ \mathrm{for\ all}\ \ K\in \mathcal{K}_{(o)}^{n,1}. \label{Main-Theory-1-ineuqlity} \end{align} 
If $\Phi$ is strictly convex, then the origin-symmetric ellipsoids are the only maximizers. }

A special case of particular interest is when $\Phi$ equals $\Phi_Q=\phi\circ h_Q$, where  $Q\in\mathcal{K}_{o}^{1,m}$ and $\phi:[0,\infty)\to[0,\infty)$ is a  convex function such that $\phi(0)=0$ and $\phi$ is strictly increasing on $[0,\infty)$. It can be checked that $\Phi_Q\in\mathcal{C}$ (see the details in Section \ref{sec-definition}).  For convenience, let $\Pi_{\phi, Q}^{m}K=\Pi_{\Phi_Q}^m K$ and  $\Pi^{m,*}_{\phi, Q}K=(\Pi^{m}_{\phi, Q}K)^*$ be the polar body of $\Pi^m_{\phi, Q}K$. For $K\in \mathcal{K}_{(o)}^{n,1}$, let  \begin{align*}\Gamma_{\phi, Q}(K)=
\frac{V_{nm}(\Pi_{\phi, Q}^{m, *}K)}{V_{n}(K)^m}.
\end{align*} 
The following higher-order Orlicz-Petty projection inequality can be obtained. When $\phi(t)=t^p$ with $p> 1$, Theorem \ref{Phi-Q-strictly convex}
reduces to the higher-order $L_p$ Petty projection inequality \cite{Haddad-ye-lp-2025}, while if $\phi(t)=t$, it recovers the $m$th higher-order Petty projection inequality \cite{Haddad-Ye-2023} (without equality characterization). 
\vskip 2mm \noindent {\bf Theorem \ref{Phi-Q-strictly convex}} {\em 
Let  $Q\in\mathcal{K}_{o}^{1,m}$ and  $\phi:[0, \infty)\rightarrow [0,\infty)$ be a convex function such that $\phi(0)=0$ and $\phi$ is strictly increasing on $[0,\infty)$.  Then, among $K\in \mathcal{K}_{(o)}^{n,1}$, $\Gamma_{\phi, Q}(K)$ is maximized at origin-symmetric ellipsoids, i.e., \begin{align*}  
\Gamma_{\phi, Q}(K)\leq \Gamma_{\phi, Q}(\ball)\ \ \mathrm{for\ all}\ \ \mathcal{K}_{(o)}^{n,1}.\end{align*} If in addition $\phi$ is strictly convex on $[0,\infty)$, then the origin-symmetric ellipsoids are the only maximizers of  $\Gamma_{\phi, Q}(K)$. }

The structure of this paper is organized as follows.   Section \ref{notation} provides the necessary background and notations required for later context. Section \ref{sec-definition} dedicates to the definition of the $m$th order Orlicz projection body and to proving that $\Pi^m_{\Phi}K \in \mathcal{K}^{n,m}_{(o)}$ for each $K\in \mathcal{K}^{n,1}_{(o)}$. The continuity and the affine invariance of the operator  $\Pi^m_{\Phi}: \mathcal{K}^{n,1}_{(o)}\rightarrow \mathcal{K}^{n,m}_{(o)}$ are established in Section \ref{section-continuity}. We would like to mention that the proof of the continuity is much more delicate for $m\ge 2$, compared to the case $m=1$ and the case for the $(L_p,Q)$-projection body \cite{Haddad-Ye-2023,Haddad-ye-lp-2025}.
We will prove Theorem \ref{Main-Theory-1} in Section \ref{sec-inequality}  and Theorem \ref{Phi-Q-strictly convex} in Section  \ref{sec-characterization-equality}. 

\section{Preliminaries and notations}\label{notation}
In this section, we will explain some basic background and necessary notation. For more details, please refer to \cite{Geometric tomography, schneider}.
 
Let $n,m, k,l\in \mathbb{N}$, where $\mathbb{N}$ is the set of positive integers. Denote by $M_{n,m}(\mathbb{R})$ the space of real $n\times m$ matrices, and by $I_n$ the $n\times n$ identity matrix.  Whenever making sense, we use $A^\mathrm{T}$, $A^{-1}$,  $\mathrm{tr} A$ and $\det A$ to mean the transpose, the inverse, the trace and the determinant of a matrix $A$, respectively. For convenience, we let  $A^\mathrm{-T}=(A^\mathrm{T})^{-1}.$ The group of invertible $n\times n$ matrices is denoted by $GL(n)\subset M_{n,n}(\mathbb{R})$, and $SL(n)$ refers to the subgroup consisting of matrices whose determinants are $1$. By $\pmb{x}{}_{\bigcdot}\, \pmb{y}$ we mean the matrix product of $\pmb{x}\in M_{n,k}(\mathbb{R})$ and $\pmb{y}\in M_{k,m}(\mathbb{R})$, and hence $\pmb{x}{}_{\bigcdot}\, \pmb{y} \in M_{n,m}(\mathbb{R})$.
We will use boldface letters, such as $\pmb{x}, \pmb{y}$ to represent element in $M_{n,m}(\mathbb{R})$, while we will use lowercase letters such as $x, y$ to represent elements in $M_{n,1}(\mathbb{R})$ and $M_{1,m}(\mathbb{R})$. For $\pmb{x},\pmb{y}\in M_{n,m}(\mathbb{R})$, let $\pmb{x}\bigcdot \pmb{y}= \mathrm{tr}(\pmb{x}^\mathrm{T} {}_{\bigcdot}\, \pmb{y}) \in \R$ and let $|\pmb{x}|=\sqrt{\pmb{x}\bigcdot \pmb{x}}$. By $V_{nm}(K)$ we mean the volume of $K\subseteq M_{n,m}(\mathbb{R})$. The origin of $M_{n,m}(\mathbb{R})$ is denoted by $o$. Note that the symbol $o$ is always for the origin whose dimension may vary depending on the context.

We often identify $M_{n,m}(\mathbb{R})$ as $\R^d$ with $d=mn$ whenever the matrix multiplication is not under consideration: a matrix $\pmb{x}\in M_{n,m}(\mathbb{R})$ will be identified as $(x_i)_i=(x_1,\cdots,x_m)$, where each $x_i\in M_{n,1}(\mathbb{R})$ is the $i$-th column of $\pmb{x}$. So notations regarding convex bodies in $\R^d$ can be easily transited to the matrix setting. For convenience, 
the Euclidean unit ball in $M_{n,m}(\mathbb{R})$ will be simply written by $B_2^{nm}$ and the unit Euclidean sphere in $M_{n,m}(\mathbb{R})$ by $S^{nm-1}$, that is, $$B_2^{nm}=\{\pmb{x}\in M_{n,m}(\mathbb{R}): \pmb{x}\bigcdot \pmb{x}\leq 1\}\ \ \mathrm{and}\ \ S^{nm-1}=\{\pmb{x}\in M_{n,m}(\mathbb{R}): \pmb{x}\bigcdot \pmb{x}=1\}.$$    

We say $K\subseteq M_{n,m}(\mathbb{R})$ is convex if $\lambda \pmb{x}+(1-\lambda)\pmb{y}\in K$  for all $\pmb{x}, \pmb{y}\in K$ and   $\lambda\in (0,1)$.  A convex set $K\subseteq M_{n,m}(\mathbb{R})$ is called a convex body if $K$ is compact  with nonempty interior. 
Denote by $\mathcal{K}^{n,m}$ the set of convex bodies in $M_{n,m}(\mathbb{R})$. With extra subscripts $o$ and $(o)$ in $\mathcal{K}^{n,m}$, we mean the set of convex bodies containing the origin $o$, and respectively, the set of convex bodies containing $o$ in their interiors. A convex body $K$ is said to be origin-symmetric if $K=-K=\{-\pmb{x}: \pmb{x}\in K\}$. We call $K\subset M_{n,1}(\mathbb{R})$ an origin-symmetric ellipsoid if $K=AB^{n}_2$ for some $A\in GL(n)$. For $E\subset M_{k,m}(\mathbb{R})$, and $\pmb{x}\in M_{n,k}(\mathbb{R})$, let $$\pmb{x}{}_{\bigcdot}   E=\{\pmb{x}{}_{\bigcdot}\, \pmb{y}: \  \pmb{y}\in E\}.$$

For $K\in\mathcal{K}^{n,m}$, the support function of $K$, denoted by $h_K: M_{n,m}(\mathbb{R})\rightarrow\mathbb{R}$, is defined by
\begin{align}\label{support-def-1}
h_K(\pmb{x})=\max\left\{\pmb{x}\bigcdot \pmb{y}:\pmb{y}\in K\right\}  \ \ \text{for}~ \pmb{x}\in M_{n,m}(\mathbb{R}).
\end{align} It can be easily checked that $h_K$ is convex on $ M_{n,m}(\mathbb{R})$ such that 
\begin{align} \label{hom-1-1}
h_K(c   \pmb{x})&=ch_K(\pmb{x})  \ \ \mathrm{for} \ \  \pmb{x}\in M_{n,m}(\mathbb{R}) \ \ \mathrm{and} \ \ c>0,\\ 
\label{h AK} 
h_{K}(A{}_{\bigcdot}\, \pmb{x})&=h_{A^{\mathrm{T}}{}_{\bigcdot}\, K}(\pmb{x}) \ \ \mathrm{for} \ \ A\in M_{n,k}(\mathbb{R})  \ \ \mathrm{and} \ \ \pmb{x}\in M_{k,m}(\mathbb{R}),\\
\label{h KB} 
h_{K}(\pmb{y}{}_{\bigcdot}\, B)&=h_{K{}_{\bigcdot}\, B^{\mathrm{T}}}(\pmb{y})\ \ \mathrm{for} \ \ B\in M_{l,m}(\mathbb{R})  \ \ \mathrm{and} \ \ \pmb{y}\in M_{n,l}(\mathbb{R}).
\end{align} It is also well-known that if a function on $ M_{n,m}(\mathbb{R})$ is convex and satisfies \eqref{hom-1-1}, then it must be the support function of a convex body in $ M_{n,m}(\mathbb{R})$, and hence uniquely determines a convex body in $ M_{n,m}(\mathbb{R})$. 
For $K_1, K_2\in \mathcal{K}^{n,m}$, we define the Hausdorff distance between $K_1$ and $K_2$ by \begin{align}
d_H(K_1, K_2) 
=\sup_{\pmb{u}\in S^{nm-1}}\Big|h_{K_1}(\pmb{u})-h_{K_2}(\pmb{u})\Big|. \label{metric-Hsdf}
\end{align} 
Let $\mathbb{N}_0=\mathbb{N}\cup\{0\}$ and  $\{K_j\}_{j\in \mathbb{N}_0}\subset \mathcal{K}^{n,m}$ be a sequence of convex bodies in $ M_{n,m}(\mathbb{R})$. We say $K_j\rightarrow K_0$ in the Hausdorff metric if $d_H(K_j,K_0)\rightarrow 0$ as $j\rightarrow \infty.$ 

For $K\in \mathcal{K}^{n,m}$, denote by $\partial K$ the boundary of $K$, {$\mathrm{int}K$ the interior of $K$, and $\nu_K(\pmb{y})$ an outer unit normal vector of $K$ at $\pmb{y}\in\partial K$. The reverse of $\nu_K$ is denoted by $\nu_K^{-1}.$ By  
 $\mathcal{H}^{nm-1}|_E$, we mean the $(nm-1)$-dimensional Hausdorff measure of $E\subset M_{n,m}(\mathbb{R})$. We often use $\mathcal{H}$ when the set $E$ and the dimension are clearly identified. In particular, we let $\,d\pmb{u}=\,d\mathcal{H}|_{S^{nm-1}}.$ Define $S_K$, the surface area measure of $K$, on $S^{nm-1}$ by $$S_K(\eta)=\mathcal{H}^{nm-1}\big(\nu_K^{-1}(\eta)\big)\ \ \mathrm{for\ each\ Borel\ set}\ \eta\subset S^{nm-1}.$$ It is well-known that, for each continuous function $f:S^{nm-1}\rightarrow \mathbb{R}$, 
\begin{align}\label{ S k and H n-1}
\int_{S^{nm-1}}f(\pmb{u})dS_K(\pmb{u})=\int_{\partial K} f\big(\nu_K(\pmb{y})\big)d\mathcal{H}^{nm-1}(\pmb{y}).
\end{align}

Let $L\in\mathcal{K}_{(o)}^{n,m}$. Its polar body, denoted by $L^*$, is defined by  
\begin{align}\label{L-star}
L^* =\left\{\pmb{x}\in M_{n,m}(\mathbb{R}):\pmb{x}\bigcdot \pmb{y}\leq1 \ ~\text{for}~\ \pmb{y}\in L \right\}.
\end{align} 
Clearly, for $L\in\mathcal{K}_{(o)}^{n,m}$, $h_{L^*}(\pmb{x})=1$ if and only if $\pmb{x}\in\partial{L}$. Moreover, $(cL)^*=(1/c)L^*$ for $c>0$, and   $(A{}_{\bigcdot}\, L)^*=A^{-\mathrm{T}}{}_{\bigcdot}\, L^*$ for $A\in GL(n)$. The volume of $L^*\in \mathcal{K}_{(o)}^{n,m}$ can be calculated by 
\begin{align*}
V_{nm}(L^*)=\frac{1}{nm}\int_{S^{nm-1}} \Big(\frac{1}{h_L(\pmb{u})}\Big)^{nm}d\pmb{u}. 
\end{align*}  

\section{The $m$th order Orlicz projection body}\label{sec-definition}
Inspired by the Orlicz projection bodies \cite{orlicz projection} and the $(L_p,Q)$-projection bodies \cite{Haddad-ye-lp-2025},  we introduce the $m$th higher-order Orlicz $L_{\Phi}$ projection body in this section.
 
Let  $\Phi:M_{1,m}(\mathbb{R})\rightarrow[0,\infty)$ be a continuous, convex function, that is, 
\begin{align*}
\Phi (\lambda z+(1-\lambda)\widetilde{z} )\leq \lambda \Phi(z) +(1-\lambda)\Phi(\widetilde{z})\ \ \mathrm{for}\ \lambda>0 \ \mathrm{and}\  z,\, \widetilde{z}\in M_{1,m}(\mathbb{R}).
\end{align*} 
Denote by $\mathcal{C}$ the set of convex functions $\Phi: M_{1,m}(\mathbb{R})\rightarrow [0, \infty)$ satisfying that   $\Phi(o)=0$, and   \begin{align} \Phi(z)+\Phi(- z)>0  \ \ \mathrm{for} \ \ z\neq o.\label{positive-z}\end{align} Particularly, the function $\Phi(z)+\Phi(-z): S^{m-1}\rightarrow \mathbb{R}$ is positive and  continuous on $S^{m-1}$. Thus,  
\begin{align}\label{mprime}
 \min_{z\in S^{m-1}} \{\Phi(z)+\Phi(-z)\}>0.
\end{align}  
Notice that \eqref{positive-z} is equivalent to that either $\Phi(z)$  or $\Phi(-z)$ is positive.  If $\Phi$ is a nonnegative convex function such that $\Phi(o)=0$, then for $z\neq o$ such that $\Phi(z)> 0$, one has  
 \begin{align}\label{Phi-ru-to-infty}
\lim_{r\rightarrow\infty}\Phi(rz) \geq \lim_{r\rightarrow \infty} r\Phi(z)=\infty,
 \end{align}  
a direct consequence of the following inequality:
\begin{align}\label{phi-rz-ge-r-phi-z}
\Phi(z)=\Phi\bigg(\frac{1}{r}rz+\Big(1-\frac{1}{r}\Big)o\bigg)\leq \frac{1}{r}\Phi(rz), \ \ \ r>1. 
 \end{align}

Our main object in this paper is the following $m$th higher-order Orlicz $L_{\Phi}$ projection body of $K\in\mathcal{K}_{(o)}^{n,1}$, denoted by $\Pi_{\Phi}^{m}K$ and often called by the $m$th order Orlicz projection body of $K$.  For $\pmb{x}=({x}_i)_i\in M_{n,m}(\mathbb{R})\setminus\{o\}$ with each ${x}_i\in M_{n,1}(\mathbb{R})$, define  $H_{\pmb{x}}:[0,\infty)\rightarrow\mathbb{R}$ by
\begin{align}\label{H-x}
H_{\pmb{x}}(t)=\int_{S^{n-1}}\Phi \bigg(t \frac{{u}^{\mathrm{T}}{}_{\bigcdot}\,  \pmb{x}}{h_K({u})}\bigg)h_K({u})dS_K({u}).
\end{align}  
\begin{definition}\label{supp-pro}
Let $\Phi\in\mathcal{C}$ and $K\in\mathcal{K}_{(o)}^{n,1}$. Define the function $h_{\Pi_{\Phi}^{m}K}(\bigcdot): M_{n,m}(\mathbb{R})\rightarrow (0,\infty)$ by $h_{\Pi_{\Phi}^{m}K}(o)=0$, and for $\pmb{x}\in M_{n,m}(\mathbb{R})\setminus\{o\},$ 
\begin{align*}
h_{\Pi_{\Phi}^{m}K}(\pmb{x})=\inf\left\{\frac{1}{t}>0: H_{\pmb{x}}(t)\leq nV_n(K)\right\}.
\end{align*}
\end{definition}

To study the properties of $h_{\Pi_{\Phi}^{m}K}(\bigcdot)$, we shall need the following proposition. 
\begin{proposition}\label{Proposition-Hx}
Let $\Phi\in\mathcal{C}$, $K\in\mathcal{K}_{(o)}^{n,1}$ and $\pmb{x}\in M_{n,m}(\mathbb{R})\setminus\{o\}$.
Then the function $H_{\pmb{x}}(\bigcdot)$ is continuous and convex on $[0,\infty)$ such that 
\begin{align}\label{h-pro1}
\lim\limits_{ t \rightarrow0^+}H_{\pmb{x}}(t) =0 \ \ \mathrm{and} \ \  
\lim\limits_{t \rightarrow\infty}H_{\pmb{x}}(t)=\infty.
\end{align} Moreover, there exists a unique $t_0\in (0, \infty)$ such that $H_{\pmb{x}}(t_0)=nV_n(K).$
\end{proposition} 
\begin{proof}
The continuity of $H_{\pmb{x}}(\bigcdot)$ follows directly from the dominated convergence theorem, the continuity of $\Phi$, and the fact that $h_K$ has positive lower bound and finite upper bound due to $K\in\mathcal{K}_{(o)}^{n,1}$; in particular, the first equality in  \eqref{h-pro1} holds. 
 
We now prove the second equality in \eqref{h-pro1}. To this end, for fixed $\pmb{x}\in M_{n, m}(\mathbb{R})$, let 
\begin{align*}
M_1=\cap_{i=1}^m  x_i^{\perp}\ \ \mathrm{and} \ \ 
M_2=\left\{z\in M_{n,1}(\mathbb{R}): \Phi\big({z^{\mathrm{T}}{}_{\bigcdot}\, \pmb{x}}\big)=0\right\},  
\end{align*} 
where $z^{\perp}$ denotes the $(n-1)$-dimensional subspace orthogonal to $z\in M_{n,1}(\mathbb{R})$. Clearly, 
$M_1\subset M_2$ because $u^{\mathrm{T}}{}_{\bigcdot}\, \pmb{x}=o$ for $u\in M_1$ and $\Phi(o)=0$.  It can be checked that $M_2$ is both closed and convex. For the convexity of $M_2$, we can check that, for  $z_1,z_2\in M_2$ and $\eta\in[0,1]$,   the convexity of $\Phi$ implies  
\begin{align*}
0\leq \Phi\left(\big(\eta z_1+(1-\eta)z_2\big)^{\mathrm{T}}{}_{\bigcdot}\, \pmb{x}\right)\leq \eta\Phi\big({z_1^{\mathrm{T}}{}_{\bigcdot}\, \pmb{x}}\big)+(1-\eta)\Phi\big({z_2^{\mathrm{T}}{}_{\bigcdot}\, \pmb{x}}\big)=0.
\end{align*}  That is,  $\eta z_1+(1-\eta)z_2\in M_2$, and thus $M_2$ is convex. For the closedness of $M_2$, let $z$ be in $\overline{M_2}$, the closure of $M_2$. Then there exists $\{z_j\}_{j\in\mathbb{N}}\subset M_2$  such that $z_j\rightarrow z$ as $j\rightarrow \infty$.  By the continuity of $\Phi$ and $z_j^{\mathrm{T}}{}_{\bigcdot}\, \pmb{x}\rightarrow z^{\mathrm{T}}{}_{\bigcdot}\, \pmb{x}$ as $j\rightarrow \infty$, we have 
\begin{align*}
\Phi\left({z^{\mathrm{T}}{}_{\bigcdot}\, \pmb{x}}\right)=\lim_{j\rightarrow \infty}\Phi\left({z_j^{\mathrm{T}}{}_{\bigcdot}\, \pmb{x}}\right)=0.
\end{align*} Therefore,  $z\in M_2$ and $M_2$ is closed.  

On the other hand, it follows from \eqref{positive-z} that, for $z\in M_{n, 1}(\mathbb{R)}\setminus  M_1$, either $\Phi\left(z^{\mathrm{T}}{}_{\bigcdot}\, \pmb{x}\right)>0,$ or $\Phi\left(-z^{\mathrm{T}}{}_{\bigcdot}\, \pmb{x}\right)>0$, or both hold true. Consequently, $o\in\partial M_2$. As  $M_2$ is closed and convex, it follows from  \cite[Theorem 1.3.2]{schneider} that there exists a support hyperplane for $M_2$ passing through $o$. In particular, $M_2$ is contained in a closed half space. This further shows that  
\begin{align}\label{M-3}
M_3=\left\{u\in S^{n-1}: \Phi(a u ^{\mathrm{T}}{}_{\bigcdot}\, \pmb{x})=0 \ \ \mathrm{for\ some}\ \ a>0 \right\}
\end{align}
is contained in a closed hemisphere of $S^{n-1}$, and hence  $S^{n-1}\setminus M_3$ contains an open hemisphere.  As $S_K(\bigcdot)$ is not concentrated on any closed hemisphere,  $S_K(S^{n-1}\setminus M_3)>0.$ Hence, for any $t\in (0, \infty),$  it follows from the fact that $\Phi$ is nonnegative that 
$$H_{\pmb{x}}(t)=\int_{S^{n-1} }\Phi\bigg(t\frac{u^{\mathrm{T}}{}_{\bigcdot}\,  \pmb{x}}{h_K(u)}\bigg)h_K(u)dS_K(u)\geq  \int_{S^{n-1} \setminus M_3}\Phi\bigg(t\frac{u^{\mathrm{T}}{}_{\bigcdot}\,  \pmb{x}}{h_K(u)}\bigg)h_K(u)dS_K(u)>0.$$ 
Note that, if $u\in S^{n-1}\setminus M_3,$ then 
$$ \lim_{t\rightarrow \infty} \Big| t\frac{u^{\mathrm{T}}{}_{\bigcdot}\,  \pmb{x}}{h_K(u)}\Big|=\infty. $$ It follows from \eqref{Phi-ru-to-infty} that $$\lim_{t\rightarrow\infty}\Phi\bigg(t\frac{u^{\mathrm{T}}{}_{\bigcdot}\,  \pmb{x}}{h_K(u)}\bigg)=\infty. $$ Combining with the continuity of $\Phi$ and Fatou's lemma, one gets 
\begin{align*}
\liminf_{t\rightarrow\infty}H_{\pmb{x}}(t)
&\ge
\liminf_{t\rightarrow\infty}\int_{S^{n-1} \setminus M_3}\Phi\bigg(t\frac{u^{\mathrm{T}}{}_{\bigcdot}\,  \pmb{x}}{h_K(u)}\bigg)h_K(u)dS_K(u)\\
&\ge\int_{S^{n-1}\setminus M_3}\liminf_{t\rightarrow\infty}\Phi\bigg(t\frac{u^{\mathrm{T}}{}_{\bigcdot}\,  \pmb{x}}{h_K(u)}\bigg)h_K(u)dS_K(u)=\infty.
\end{align*}

We now prove the convexity of $H_{\pmb{x}}(\bigcdot)$. To this end, for $t_1,t_2\in [0,\infty)$ and $\eta\in[0,1]$, the convexity of $\Phi$ implies that
\begin{align*}
\Phi\bigg(\big(\eta t_1+(1-\eta)t_2\big)\frac{{u}^{\mathrm{T}}{}_{\bigcdot}\,  \pmb{x}}{ h_K({u})}\bigg)
\leq \eta \Phi\bigg(t_1
\frac{{u}^{\mathrm{T}}{}_{\bigcdot}\,  \pmb{x}}{h_K({u})}\bigg)+(1-\eta)\Phi\bigg(t_2
\frac{{u}^{\mathrm{T}}{}_{\bigcdot}\,  \pmb{x}}{h_K({u})}\bigg).
\end{align*}
Integrating the above inequality with respect to $h_KdS_K$, we obtain $$H_{\pmb{x}}(\eta t_1+(1-\eta)t_2)\leq \eta H_{\pmb{x}}(t_1)+(1-\eta)H_{\pmb{x}}(t_2),$$  which gives the convexity of $H_{\pmb{x}}(\bigcdot)$.  

Finally, the fact that $H_{\pmb{x}}(\bigcdot)$ is continuous on $[0, \infty)$ implies that the set $$H_{\pmb{x}}^{-1}(\{0\})=\{t\in [0, \infty): H_{\pmb{x}}(t)=0\}$$ is a closed set. Since $\lim_{t\rightarrow \infty} H_{\pmb{x}}(t)=\infty$, then $$b_0=\sup\{t\in [0, \infty): H_{\pmb{x}}(t)=0\}<\infty.$$ Therefore, $H_{\pmb{x}}(\bigcdot)$ is strictly increasing on $[b_0, 
\infty).$ To see this, assume that there exist $b_0\leq t_1<t_2<\infty$ such that $H_{\pmb{x}}(t_1)=H_{\pmb{x}}(t_2).$ By the  facts that $H_{\pmb{x}}(0)=0$ and $H_{\pmb{x}}(\bigcdot)$ is a convex function, one has   \begin{align*}
H_{\pmb{x}}(t_1)=H_{\pmb{x}}\bigg(\frac{t_1}{t_2}t_2+\bigg(1-\frac{t_1}{t_2}\bigg)0\bigg)\leq \frac{t_1}{t_2}H_{\pmb{x}}(t_2)<H_{\pmb{x}}(t_1),
\end{align*} a contradiction. This shows that $H_{\pmb{x}}(\bigcdot)$ is strictly increasing on $[b_0, \infty).$ In particular, there exists a unique $t_0=t_0(\pmb{x})\in (0, \infty)$ such that $H_{\pmb{x}}(t_0)=nV_n(K).$  This completes the proof. \end{proof}

Based on Definition \ref{supp-pro} and Proposition \ref{Proposition-Hx}, for $\Phi\in\mathcal{C}$ and $K\in\mathcal{K}_{(o)}^{n,1}$, by letting $t=\frac{1}{h_{\Pi_{\Phi}^{m}K}(\pmb{x})}$, one has, for $\pmb{x}\in M_{n,m}(\mathbb{R})\setminus \{o\}$,  
\begin{align}\label{lemma equiv of pro body}
H_{\pmb{x}}\bigg(\frac{1}{h_{\Pi_{\Phi}^{m}K}(\pmb{x})}\bigg)=\int_{S^{n-1}}\Phi \bigg( \frac{{u}^{\mathrm{T}}{}_{\bigcdot}\,  \pmb{x}}{h_{\Pi_{\Phi}^{m}K}(\pmb{x})h_K({u})}\bigg)h_K({u})dS_K({u})=nV_n(K).
\end{align} 
We now prove that the function $h_{\Pi_{\Phi}^{m}K}(\bigcdot): M_{n,m}(\mathbb{R})\rightarrow [0, \infty)$ is a convex function with positive homogeneity of degree $1$, and hence it uniquely determines a convex body in $M_{n, m}(\mathbb{R})$. 

\begin{proposition} Let $\Phi\in\mathcal{C}$ and $K\in\mathcal{K}_{(o)}^{n,1}$. Then, $\Pi_{\Phi}^{m}K\in \mathcal{K}_{(o)}^{n,m}$, i.e., $\Pi_{\Phi}^{m}K$ is a convex body in $M_{n, m}(\mathbb{R})$ containing the origin in its interior. 
\end{proposition} 

\begin{proof} We first show that $h_{\Pi_{\Phi}^{m}K}(\bigcdot): M_{n,m}(\mathbb{R})\rightarrow [0, \infty)$ is the support function of $\Pi_{\Phi}^{m}K$.
It can be easily checked from \eqref{lemma equiv of pro body} that $$h_{\Pi_{\Phi}^{m}K}(a\pmb{x})=ah_{\Pi_{\Phi}^{m}K}(\pmb{x}) \ \ \mathrm{for\  all} \ \pmb{x}\in M_{n,m}(\mathbb{R}) \ \ \mathrm{and} \ \ a>0.$$ On the other hand, 
for any $\pmb{x}, \widetilde{\pmb{x}}\in M_{n,m}(\mathbb{R})\setminus \{o\}$,
let 
$$\lambda=h_{\Pi_{\Phi}^{m}K}(\pmb{x}) \ \ \mathrm{and} \ \  \widetilde{\lambda}=h_{\Pi_{\Phi}^{m}K}(\widetilde{\pmb{x}}). $$ 
It follows from  \eqref{lemma equiv of pro body} that   
$$H_{\pmb{x}}\Big(\frac{1}{\lambda}\Big)=nV_n(K)\ \ \mathrm{and}\ \ H_{\widetilde{\pmb{x}}}\Big(\frac{1}{\widetilde{\lambda}}\Big)=nV_n(K).$$
As  $\Phi$ is convex, one has  
\begin{align*}
\Phi\bigg(\frac{u^{\mathrm{T}}{}_{\bigcdot}\,  (\pmb{x}+\widetilde{\pmb{x}})}{(\lambda+\widetilde{\lambda}) h_K(u)}\bigg)
&\leq\frac{\lambda}{\lambda+\widetilde{\lambda}}\Phi\bigg(\frac{u^{\mathrm{T}}{}_{\bigcdot}\,  \pmb{x}}{\lambda h_K(u)}\bigg)+
\frac{\widetilde{\lambda}}{\lambda+\widetilde{\lambda}}\Phi\bigg(\frac{u^{\mathrm{T}}{}_{\bigcdot}\,  \widetilde{\pmb{x}}}{\widetilde{\lambda} h_K(u)}\bigg).
\end{align*}
Integrating both sides with respect to $h_KdS_K$, one gets
\begin{align*}
\int_{S^{n-1}}\Phi\bigg(\frac{u^{\mathrm{T}}{}_{\bigcdot}\,  (\pmb{x}+\widetilde{\pmb{x}})}{(\lambda+\widetilde{\lambda}) h_K(u)}\bigg)h_K(u)dS_K(u)\leq \frac{\lambda}{\lambda+\widetilde{\lambda}} H_{\pmb{x}}\Big(\frac{1}{\lambda}\Big)+\frac{\widetilde{\lambda}}{\lambda+\widetilde{\lambda}}H_{\widetilde{\pmb{x}}}\Big(\frac{1}{\widetilde{\lambda}}\Big)=nV_n(K).
\end{align*}
Definition \ref{supp-pro} then implies that $$h_{\Pi_{\Phi}^{m}K}(\pmb{x}+\widetilde{\pmb{x}})\leq \lambda+\widetilde{\lambda}=h_{\Pi_{\Phi}^{m}K}(\pmb{x})+h_{\Pi_{\Phi}^{m}K}(\widetilde{\pmb{x}}).$$ This concludes that $h_{\Pi_{\Phi}^{m}K}(\bigcdot): M_{n,m}(\mathbb{R})\rightarrow [0,\infty)$ is a convex function with positive homogeneity of degree $1$, and hence it uniquely determines a convex body $\Pi_{\Phi}^{m}K$. 

Moreover, by Definition \ref{supp-pro} and Proposition \ref{Proposition-Hx},   ${h_{\Pi_{\Phi}^{m}K}(\pmb{x})}>0$ for any $\pmb{x}\in M_{n,m}(\mathbb{R})\setminus \{o\}$. Thus, $\Pi_{\Phi}^{m}K$ is a convex body in $M_{n, m}(\mathbb{R})$ that contains the origin in its interior. 
\end{proof}

Let us pause here to mention that, for $Q\in\mathcal{K}_{o}^{1,m}$,  the support function $h_Q(\bigcdot)$ belongs to $\mathcal{C}$. In fact, $h_Q(o)=0$ holds for any $Q\in\mathcal{K}_{o}^{1,m}.$ On the other hand, \eqref{positive-z} clearly holds if $o$ is in the interior of $Q$. When $o\in \partial{Q}$, there exists a support hyperplane, say $\mathbf{H}$, passing through $o$, and then $Q$ is contained in a closed half space given by such a hyperplane, say $\mathbf{H}^{-}$. It is clear that if $z\in \mathbf{H}^{-}\setminus\{o\},$ then $h_{Q}(z)>0$, which further leads to \eqref{positive-z} in the case when $o\in \partial Q.$ Similarly, it can also be proved that $h_Q^p(\bigcdot)\in \mathcal{C}$ for $Q\in  \mathcal{K}_{o}^{1,m}$ and $p\geq 1.$ In a more general setting, for $Q\in\mathcal{K}_{o}^{1,m}$, 
 one can consider the convex function 
$\Phi_Q=\phi\circ h_Q: M_{1, m}(\mathbb{R})\rightarrow [0, \infty)$, where $\phi:[0, \infty)\rightarrow [0,\infty)$ is a convex function such that $\phi(0)=0$ and $\phi$ is strictly increasing on $[0,\infty)$. As proved above, one can easily see $\Phi_Q \in \mathcal{C}.$ This gives a special case of $\Pi_{\Phi}^mK$ for  $K\in\mathcal{K}_{(o)}^{n,1}$, denoted by $\Pi_{\phi, Q}^{m}K$, with its support function given by: for $\pmb{x}\in M_{n, m}(\R)$, 
\begin{align}\label{specical-case-1}
h_{\Pi_{\phi, Q}^{m}K}(\pmb{x})=\inf\left\{\frac{1}{t}>0: \int_{S^{n-1}} \Phi_Q \bigg(t \frac{{u}^{\mathrm{T}}{}_{\bigcdot}\,  \pmb{x}}{h_K({u})}\bigg)h_K({u})dS_K({u})\leq nV_n(K)\right\}.
\end{align} 
By letting $\phi(t)=t^p$, $p>1$, one gets 
$
\Pi_{\phi,Q}^mK=\frac{\Pi_{p,Q}K}{(nV_n(K))^{\frac{1}{p}}},
$
where $\Pi_{p,Q}K$ is the 
 $(L_p,Q)$-projection body \cite[Definition 1.2]{Haddad-ye-lp-2025}. Moreover, by letting $\phi(t)=t$ and $Q=\Delta_m$, one gets 
 $
\Pi_{\phi,Q}^mK=\frac{\Pi^mK}{nV_n(K)},
 $
 where  $\Pi^mK$ is the $m$th higher-order projection body \cite[Definition 1.2]{Haddad-Ye-2023}. Meanwhile, when $m=1$, $\Pi^m_{\Phi}K$ reduces to the Orlicz projection body \cite[P.228]{orlicz projection}, and reduces, up to a constant, to the asymmetric $L_p$ projection bodies (see, e.g., \cite[P.8]{general lp affine iso ine}, \cite[P.16]{ludwig 2005}), the (symmetric) $L_p$ projection body \cite[P.116]{lp proj ineq 1} and the complex projection body \cite{complex projection body}.

\section{Continuity and affine invariance of the $m$th order Orlicz projection body} \label{section-continuity}
In this section, we will prove the continuity and affine invariance for the $m$th order Orlicz projection body. These properties are essential for the proof of the inequalities related to the $m$th order Orlicz  projection body in Section \ref{sec-inequality}.
 
Recall that for $K_i, K\in\mathcal{K}_{(o)}^{n,1}$, $K_i\rightarrow K$ in the Hausdorff metric if $d_H(K_i, K)\rightarrow 0$, where $d_H(\bigcdot, \bigcdot)$ is defined in \eqref{metric-Hsdf}. We are in the position to show our first continuity result. 
\begin{proposition}\label{continuity} 
 Let $\Phi\in\mathcal{C}$ and $K_0\in\mathcal{K}_{(o)}^{n,1}$.  
Suppose that $\{K_j\}_{j\in\mathbb{N}}\subset\mathcal{K}_{(o)}^{n,1}$ is a sequence of convex bodies satisfying $K_j\rightarrow K_0$ in the Hausdorff metric. Then,  $$ \Pi_{\Phi}^{m}K_j \rightarrow  \Pi_{\Phi}^{m}K_0 \ \ \mathrm{as} \ \  j\rightarrow \infty$$ 
 in the Hausdorff metric.  
\end{proposition}

 \begin{proof}To prove $\Pi_{\Phi}^{m}K_j \rightarrow  \Pi_{\Phi}^{m}K_0$ as $j\rightarrow \infty$, it is equivalent to prove that 
$$h_{\Pi_{\Phi}^{m}K_j}\rightarrow h_{\Pi_{\Phi}^{m}K_0} \ \ \mathrm{uniformly\ \ on}\ \ S^{nm-1}.$$   For $\pmb{\theta}=(\theta_i)_i\in S^{nm-1}$ and $j\in\mathbb{N}_0$, it is convenient to let   $$\lambda_j(\pmb{\theta})=h_{\Pi_{\Phi}^{m}{K_j}}(\pmb{\theta}).$$

We first show that, for each $\pmb{\theta}\in S^{nm-1}$,  the sequence $\{\lambda_j(\pmb{\theta})\}_{j\in\mathbb{N}}$ has a finite upper bound.  Assume to the contrary that $\lim_{j\rightarrow \infty}\lambda_j(\pmb{\theta})=\infty$.   For $j\in\mathbb{N}_0,$ let 
\begin{align*}
y_{j,\pmb{\theta}}(u)= \frac{u^{\mathrm{T}}{}_{\bigcdot}\,  \pmb{\theta}}{\lambda_j(\pmb{\theta}) h_{K_j}(u)} \ \ \mathrm{for}\ \ u\in S^{n-1}.
\end{align*} 
Since $\{K_j\}_{j\in \mathbb{N}_0}\subset \mathcal{K}_{(o)}^{n,1}$ and $K_j\rightarrow K_0$ as $j\rightarrow \infty$, we have $h_{K_j}\rightarrow h_{K_0}$ uniformly on $S^{n-1}$. Moreover, there exist constants $c_1,c_2>0$ such that  $c_1\leq h_{K_j}(u)\leq c_2$  for all $j\in\mathbb{N}_0$ and $u\in S^{n-1}$. It can be easily checked that $|u^{\mathrm{T}}{}_{\bigcdot}\, \pmb{\theta}|\leq 1$ holds for $u\in S^{n-1}$ and $\pmb{\theta}\in S^{nm-1}$, and hence $u^{\mathrm{T}}{}_{\bigcdot}\, \pmb{\theta}\in B_2^m$. It is easily checked that $y_{j, \pmb{\theta}}(\bigcdot) \rightarrow o$  uniformly on $S^{n-1},$ namely,  $|y_{j, \pmb{\theta}}(\bigcdot )|\rightarrow 0$ uniformly on $S^{n-1}.$ 
As $\lambda_j(\pmb{\theta})\rightarrow \infty$, there exists $j_0\in \mathbb{N}$ such that $\lambda_j(\pmb{\theta})>1$ for all $j\geq j_0$, and hence $|y_{j, \pmb{\theta}}(u)|\leq \frac{1}{c_1}$ holds for all $j\geq j_0$ and $u\in \sphere$. Consequently, $\{y_{j, \pmb{\theta}}(u): j\in \mathbb{N}\ \mathrm{and}\ u\in\sphere\}$ is contained in a compact set, say $\Omega$, and in particular, $\Phi$ is uniformly continuous on $\Omega$. Thus, for any  $\epsilon>0$, there exists $\delta(\epsilon)>0$, such that, for all $y_1, y_2\in \Omega$ satisfying $|y_1-y_2|<\delta(\epsilon)$, one has  $|\Phi(y_1)-\Phi(y_2)|<\epsilon$.   The uniform convergence of  $y_{j, \pmb{\theta}}(\bigcdot)\rightarrow o$ on $S^{n-1}$ yields the existence of $j_1(\epsilon)\in \mathbb{N}$, such that, for all $j>j_1(\epsilon)$, the following holds: $$|y_{j, \pmb{\theta}}(u)|\leq \delta(\epsilon)  \ \  \mathrm{for\ any} \ \ u\in S^{n-1}. $$ In summary, we have shown that, for any   $\epsilon>0$,  there exists $j_1(\epsilon)\in \mathbb{N}$, such that, for all $j>j_1(\epsilon)$,  $$|\Phi(y_{j, \pmb{\theta}}(u)) |<\epsilon, \ \ \mathrm{for\ any}\ \ u\in S^{n-1}.$$ Thus, $\Phi(y_{j, \pmb{\theta}}(\bigcdot))\rightarrow 0$ uniformly on $S^{n-1}$. Together with \eqref{lemma equiv of pro body} and e.g., \cite[(2.5)]{zbc-hh-dpy-2018}, one gets,   
\begin{align}
1
&=\lim_{j\rightarrow\infty}\frac{1}{nV_n(K_j)}\int_{S^{n-1}}\Phi\bigg(\frac{u^{\mathrm{T}}{}_{\bigcdot}\,  \pmb{\theta}}{\lambda_j(\pmb{\theta}) h_{K_j}(u)}\bigg)h_{K_j}(u)dS_{K_j}(u) \nonumber \\
&=\frac{1}{nV_n(K)}\int_{S^{n-1}}\lim_{j\rightarrow\infty}\Phi\big(y_{j, \pmb{\theta}}(u)\big)h_{K_j}(u)dS_{K_0}(u)=0. \label{unif-conv-Phi-1}
\end{align} This is a contradiction, and therefore  
$\{\lambda_j(\pmb{\theta})\}_{j\in\mathbb{N}}$ has a finite upper bound for any given $\pmb{\theta}\in S^{nm-1}$.

We now show that, for each $\pmb{\theta}\in S^{nm-1}$, $\{\lambda_j(\pmb{\theta}) \}_{j\in\mathbb{N}}$  has a positive lower bound by an argument of contradiction. To this end, assume that there exists a convergent subsequence, which is again denoted by $\{\lambda_j(\pmb{\theta}) \}_{j\in\mathbb{N}}$, such that $\lim_{j\rightarrow \infty}\lambda_j(\pmb{\theta}) =0$. As  $\pmb{\theta}\in S^{nm-1}$, there  exists  a $k_0\in\{1,\cdots,m\}$, such that $\theta_{k_0}\ne o.$ Let 
\begin{align*}
\Sigma_{k_0}^+=\{u\in S^{n-1}: u\bigcdot \theta_{k_0}>0\}\ \ \mathrm{and} \ \ \Sigma_{k_0}^{-} =\{u\in S^{n-1}: u\bigcdot \theta_{k_0}\leq 0\}.\end{align*}
Let $z_{\pmb{\theta}, 0}^+=(z^{1, +}_{\pmb{\theta}, 0},\cdots,z^{m, +}_{\pmb{\theta}, 0})$ and $z_{\pmb{\theta}, 0}^{-}=(z^{1, -}_{\pmb{\theta}, 0},\cdots,z^{m, -}_{\pmb{\theta}, 0})$, where for each $i\in \{1, \cdots, m\},$
\begin{align*}
z^{i, +}_{\pmb{\theta}, 0}=\int_{\Sigma_{k_0}^+} u\bigcdot \theta_i dS_{K_0}(u)\ \ \mathrm{and} \ \ z^{i, -}_{\pmb{\theta}, 0}=\int_{\Sigma_{k_0}^{-}} u\bigcdot \theta_i dS_{K_0}(u).    
\end{align*} As $S_{K_0}$ is not concentrated on any closed hemisphere, \begin{align*}
z^{k_0, +}_{\pmb{\theta}, 0}&=\int_{\Sigma_{k_0}^+} u\bigcdot \theta_{k_0} dS_{K_0}(u)>0\ \ \mathrm{and} \ \ z^{k_0, -}_{\pmb{\theta}, 0}=\int_{\Sigma_{k_0}^{-}} u\bigcdot \theta_{k_0} dS_{K_0}(u)<0,
\end{align*} 
which yields that $z_{\pmb{\theta}, 0}^{+}\ne o$ and $z_{\pmb{\theta}, 0}^{-}\ne o$.
Note that $S_{K_0}$ has its centroid at the origin, namely, $  \int_{S^{n-1}} u  dS_{K_0}(u)=o$ (see, e.g., \cite[(5.30)]{schneider}).  Thus, for any $i\in \{1, \cdots, m\},$ $$ z^{i, -}_{\pmb{\theta}, 0}+z^{i, +}_{\pmb{\theta}, 0}=\int_{\Sigma_{k_0}^{-}} u\bigcdot \theta_i dS_{K_0}(u) + \int_{\Sigma_{k_0}^{+}} u\bigcdot \theta_i dS_{K_0}(u)=\bigg(
\int_{S^{n-1}} u dS_{K_0}(u)\bigg) \bigcdot \theta_i=0.$$ This shows that $z^{i, -}_{\pmb{\theta}, 0}=-z^{i, +}_{\pmb{\theta}, 0}$ for $i\in \{1, \cdots, m\},$ and hence 
$z_{\pmb{\theta}, 0}^{-}=-z_{\pmb{\theta}, 0}^+.$  Similarly, for each $j\in \mathbb{N}$, let $z^{+} _{\pmb{\theta}, j}=(z^{1, +} _{\pmb{\theta}, j}, \cdots, z^{m, +} _{\pmb{\theta}, j})$ and $z^{-} _{\pmb{\theta}, j}=(z^{1, -} _{\pmb{\theta}, j}, \cdots, z^{m, -} _{\pmb{\theta}, j})$, where for each $i\in \{1, \cdots, m\}$,  \begin{align*}
z^{i, +} _{\pmb{\theta}, j} = \int_{\Sigma_{k_0}^+} u\bigcdot \theta_i dS_{K_j}(u)\ \ \ \mathrm{and} \ \ \ z^{i, -} _{\pmb{\theta}, j}  = \int_{\Sigma_{k_0}^{-}} u\bigcdot \theta_i dS_{K_j}(u).    
\end{align*}  Again, for all $i, j$,  we can prove that $ z^{i, -} _{\pmb{\theta}, j} =-z^{i, +} _{\pmb{\theta}, j}$, and then $z^{-} _{\pmb{\theta}, j}=-z^{+} _{\pmb{\theta}, j}.$   Also, for all $j\in \mathbb{N}$, $z_{\pmb{\theta}, j}^+\neq o$ and $z_{\pmb{\theta}, j}^{-}\neq o$, as $S_{K_j}$ is not concentrated on any closed hemisphere.  

It follows from \eqref{phi-rz-ge-r-phi-z} and Jensen's inequality (see, e.g., \cite[Theorem 3.9.3]{Niculescu-Persson-convex-functions-and-their-applications}) that
\begin{align}\label{Phi-uxk}
\nonumber 1 &=\frac{1}{nV_n(K_j)} \int_{S^{n-1}}\Phi\bigg(\frac{u^{\mathrm{T}}{}_{\bigcdot}\,  \pmb{\theta}}{\lambda_j(\pmb{\theta}) h_{K_j}(u)}\bigg)h_{K_j}(u)dS_{K_j}(u)\\ 
&= \frac{1}{nV_n(K_j)} \int_{\Sigma_{k_0}^+}\Phi\bigg(\frac{u\bigcdot \theta_1}{\lambda_j(\pmb{\theta}) h_{K_j}(u)},\cdots,  \frac{u\bigcdot \theta_m}{\lambda_j(\pmb{\theta}) h_{K_j}(u)} \bigg)h_{K_j}(u) {d}S_{K_j}(u) \nonumber  \\ &\ \ \ \ +\frac{1}{nV_n(K_j)} \int_{\Sigma_{k_0}^{-}}\Phi\bigg(\frac{u\bigcdot \theta_1}{\lambda_j(\pmb{\theta}) h_{K_j}(u)},\cdots,  \frac{u\bigcdot \theta_m}{\lambda_j(\pmb{\theta}) h_{K_j}(u)} \bigg)h_{K_j}(u) {d}S_{K_j}(u) \nonumber  \\  &\ge\Phi\bigg( \frac{1}{nV_n(K_j)} \int_{\Sigma_{k_0}^+} \frac{u\bigcdot \theta_1}{\lambda_j(\pmb{\theta})} {d}S_{K_j}(u),\cdots,  \frac{1}{nV_n(K_j)} \int_{\Sigma_{k_0}^+}\frac{u\bigcdot \theta_m}{\lambda_j(\pmb{\theta})} {d}S_{K_j}(u)\bigg) \nonumber \\ &\ \ \ \ + \Phi\bigg( \frac{1}{nV_n(K_j)}\int_{\Sigma_{k_0}^{-}} \frac{u\bigcdot \theta_1}{\lambda_j(\pmb{\theta})}  {d}S_{K_j}(u),\cdots,  \frac{1}{nV_n(K_j)} \int_{\Sigma_{k_0}^{-}}\frac{u\bigcdot \theta_m}{\lambda_j(\pmb{\theta}) }  {d}S_{K_j}(u)\bigg) \nonumber \\ &=\Phi\bigg(\frac{z_{\pmb{\theta}, j}^+}{nV_n(K_j) \lambda_j(\pmb{\theta})} \bigg)+\Phi\bigg(\frac{- z_{\pmb{\theta}, j}^{+}}{nV_n(K_j) \lambda_j(\pmb{\theta})} \bigg). \end{align}  On the other hand, the function $$g_0(u)=\frac{|u\bigcdot \theta_{k_0}|+u\bigcdot \theta_{k_0}}{2}\ \ \ \mathrm{for} \ \  u\in S^{n-1}$$  is continuous on  $S^{n-1}$, such that, $g_0(u)=u\bigcdot \theta_{k_0}$ if $u\in \Sigma_{k_0}^+$, and $g_0(u)=0$ if $u\in \Sigma_{k_0}^{-}.$
Thus 
\begin{align*}
z_{\pmb{\theta}, j}^{k_0,+}=\int_{\Sigma_{k_0}^+}g_0(u) \,d S_{K_j}(u)
= \int_{S^{n-1}} g_0(u) \,d S_{K_j}(u) \ \ \mathrm{for \,\, each}\, \, j\in \mathbb{N}_0.
\end{align*} As $S_{K_j}\rightarrow S_{K_0}$ weakly on $S^{n-1},$ one gets, 
 \begin{align}
\lim_{j\rightarrow\infty} z_{\pmb{\theta}, j}^{k_0,+} 
=\lim_{j\rightarrow\infty}  \int_{S^{n-1}} g_0(u) \,d S_{K_j}(u)=\int_{S^{n-1}} g_0(u) \, d S_{K_0}(u)=z_{\pmb{\theta}, 0}^{k_0,+}>0. \label{inequality>0}
\end{align} By \eqref{inequality>0}, without loss of generality, one can get, for each $j\in \mathbb{N}$,  \begin{align}
    |z_{\pmb{\theta}, j}^+|\geq z_{\pmb{\theta}, j}^{k_0,+} \ge \frac{z_{\pmb{\theta},0}^{k_0, +}}{2}>0. \label{lower bound z+}
\end{align} 
Recall that, for each fixed $\pmb{\theta}\in S^{mn-1}$, $\lim_{j\rightarrow \infty} \lambda_j(\pmb{\theta})=0$. Together with \eqref{lower bound z+} and $\lim_{j\rightarrow \infty} V_n(K_j)=V_n(K)>0$, one gets \begin{align*}
 \lim_{j\rightarrow \infty } \frac{|z_{\pmb{\theta}, j}^{+}|}{nV_n(K_j)\lambda_j(\pmb{\theta})}  =\infty.   
\end{align*} 
By \eqref{mprime}, \eqref{phi-rz-ge-r-phi-z} and \eqref{Phi-uxk}, one has
\begin{align*}
1
&
\ge\lim_{j\rightarrow\infty}\bigg[\Phi\bigg(\frac{z_{\pmb{\theta}, j}^+}{nV_n(K_j)\lambda_j(\pmb{\theta})} \bigg)+\Phi\bigg(\frac{z_{\pmb{\theta}, j}^{-}}{nV_n(K_j)\lambda_j(\pmb{\theta})} \bigg)\bigg]\\
&\ge \lim_{j\rightarrow\infty}\bigg\{\frac{|z_{\pmb{\theta}, j}^{+}|}{nV_n(K_j)\lambda_j(\pmb{\theta})}  \bigg[\Phi\bigg(\frac{z_{\pmb{\theta}, j}^+ }{|z_{\pmb{\theta}, j}^+|}\bigg)+ 
 \Phi\bigg(\frac{-z_{\pmb{\theta}, j}^{+}}{|z_{\pmb{\theta}, j}^+|}\bigg)\bigg] \bigg\} \\
 &\ge \min \left\{\Phi(z)+\Phi(-z):z\in S^{m-1}\right\} 
 \times \lim_{j\rightarrow \infty } \frac{|z_{\pmb{\theta}, j}^{+}|}{nV_n(K_j)\lambda_j(\pmb{\theta})}  =\infty.
\end{align*}
This is a contradiction. Thus, for each $\pmb{\theta}\in S^{nm-1}$, $\{\lambda_j(\pmb{\theta}) \}_{j\in\mathbb{N}}$  has a positive lower bound. 

Let $\{\lambda_{j_k}(\pmb{\theta})\}_{k\in\mathbb{N}}$ be an arbitrary subsequence of $\{\lambda_j(\pmb{\theta})\}_{j\in\mathbb{N}}$. As $\{\lambda_{j_k}(\pmb{\theta})\}_{k\in\mathbb{N}}$  is a bounded sequence,  one can find a convergent subsequence of $\{\lambda_{j_k}(\pmb{\theta})\}_{k\in\mathbb{N}}$, say $\{\lambda_{j_{k_l}}(\pmb{\theta})\}_{l\in\mathbb{N}}$, such that \begin{align} 
\lim_{l\rightarrow \infty}\lambda_{j_{k_l}}(\pmb{\theta}) =\widetilde{\lambda}(\pmb{\theta})\in (0, \infty). \label{limit-lambda-1-1} \end{align} 
From \eqref{lemma equiv of pro body}, for any
$\pmb{\theta}\in S^{nm-1}$, one has, 
\begin{align*}
nV_n(K_0)&=\lim_{l\rightarrow \infty}nV_n(K_{j_{k_l}})\\ &
\nonumber=\lim_{l\rightarrow \infty} \int_{S^{n-1}}\Phi\bigg(\frac{u^{\mathrm{T}}{}_{\bigcdot}\,  \pmb{\theta}}{\lambda_{j_{k_l}}(\pmb{\theta})h_{K_{j_{k_l}}}(u)}\bigg)h_{K_{j_{k_l}}}(u)dS_{K_{j_{k_l}}}(u)\\ &
=\int_{S^{n-1}}\lim_{l\rightarrow \infty}  \Phi\bigg(\frac{u^{\mathrm{T}}{}_{\bigcdot}\,  \pmb{\theta}}{\lambda_{j_{k_l}}(\pmb{\theta})h_{K_{j_{k_l}}}(u)}\bigg)h_{K_{j_{k_l}}}(u)dS_{K_{0}}(u)\\ &
=\int_{S^{n-1}}   \Phi\bigg(\frac{u^{\mathrm{T}}{}_{\bigcdot}\,  \pmb{\theta}}{\widetilde{\lambda}(\pmb{\theta})h_{K_{0}}(u)}\bigg)h_{K_{0}}(u)dS_{K_{0}}(u),
\end{align*} 
where the third equality follows from the exactly same reasons as in \eqref{unif-conv-Phi-1}, and the last equality follows from \eqref{limit-lambda-1-1} and the uniform convergence of $h_{K_j}\rightarrow h_{K_0}>0$ on $S^{n-1}$. 

In conclusion, we have shown that  \begin{align*}  nV_n(K_0) 
=\int_{S^{n-1}}   \Phi\bigg(\frac{u^{\mathrm{T}}{}_{\bigcdot}\,  \pmb{\theta}}{\widetilde{\lambda}(\pmb{\theta})h_{K_{0}}(u)}\bigg)h_{K_{0}}(u)dS_{K_{0}}(u),
\end{align*} 
and hence $\widetilde{\lambda}(\pmb{\theta})=h_{\Pi_{\Phi}^{m}K_0}(\pmb{\theta})$. This further shows that, for each $\pmb{\theta}\in S^{nm-1}$, $$\lim_{j\rightarrow\infty}\lambda_{j}(\pmb{\theta})=h_{\Pi_{\Phi}^{m}K_0}(\pmb{\theta}),$$ due to the arbitrariness of $\{\lambda_{j_k}(\pmb{\theta})\}_{k\in \mathbb{N}}.$ It is well-known  (see \cite[Theorem 1.8.15]{schneider}) that, for support functions on $S^{nm-1}$, the pointwise convergence is equivalent to the uniform convergence. Thus, we conclude that $h_{\Pi_{\Phi}^{m}K_{j}}\rightarrow h_{\Pi_{\Phi}^{m}K_0}$ uniformly on $S^{nm-1}$, and then $\Pi_{\Phi}^{m}K_{j}\rightarrow \Pi_{\Phi}^{m}K_{0}$ as desired. 
\end{proof}
 
The sequence $\{\Phi_j\}_{j\in\mathbb{N}}\subset\mathcal{C}$ is said to be convergent to $\Phi\in\mathcal{C}$, if \begin{align}
    \sup_{z\in G}|\Phi_j(z)-\Phi(z)|\rightarrow 0 \label{def-func-conv-1}
\end{align} for every nonempty compact set $G\subset M_{1,m}(\mathbb{R})$. We have the following continuity result. 
\begin{proposition}\label{continuity-Phi-i} 
 Let $\Phi_0\in\mathcal{C}$ and $K_0\in\mathcal{K}_{(o)}^{n,1}$. Suppose that  $\{\Phi_j\}_{j\in\mathbb{N}}\subset\mathcal{C}$ is a sequence such that $\Phi_j\rightarrow \Phi_0$. Then,   $$\Pi_{\Phi_j}^m{K_0}\rightarrow  \Pi_{\Phi_0}^{m}{K_0} \ \ \mathrm{as} \ \  j\rightarrow \infty $$  in the Hausdorff metric. 
\end{proposition}
\begin{proof}  For each $\pmb{\theta}\in S^{nm-1}$, let $\widehat{\lambda}_j(\pmb{\theta}) =h_{\Pi_{\Phi_j}^{m}{K_0}}(\pmb{\theta})>0$ for each $j\in\mathbb{N}_0$. 
By \eqref{lemma equiv of pro body}, one gets
\begin{align}\label{about phi j}
\int_{S^{n-1}}\Phi_{j}\bigg(\frac{u^{\mathrm{T}}{}_{\bigcdot}\,  \pmb{\theta}}{\widehat{\lambda}_{j}(\pmb{\theta}) h_{K_0}(u)}\bigg)h_{K_0}(u)dS_{K_0}(u)=nV_n(K_0).
\end{align} We now claim that $\{\widehat{\lambda}_j(\pmb{\theta})\}_{j\in \mathbb{N}}$ has a finite upper bound and a positive lower bound. We prove the upper bound by an argument of contradiction. To this end, assume that $\lim_{j\rightarrow\infty}\widehat{\lambda}_j(\pmb{\theta})=\infty$. Then there exists $\widehat{j}_0\in\mathbb{N}$, such that, $\widehat{\lambda}_j(\pmb{\theta})>1$ for all $j>\widehat{j}_0$. For   $\pmb{\theta}\in S^{nm-1}$, $u\in S^{n-1}$ and $j\in\mathbb{N}_0$, let  
\begin{align*}
\widehat{y}_{j,\pmb{\theta}}(u)= \frac{u^{\mathrm{T}}{}_{\bigcdot}\,  \pmb{\theta}}{\widehat{\lambda}_j(\pmb{\theta}) h_{K_0}(u)}.
\end{align*} It can be checked that, for all $u\in S^{n-1}$,  $|u^{\mathrm{T}}{}_{\bigcdot}\, \pmb{\theta}|\leq 1$. Note that, as $K_0\in\mathcal{K}_{(o)}^{n,1}$,   there exist constants $c_1, c_2>0$ such that $c_1\leq h_{K_0}(u)\leq c_2$. Thus, $|\widehat{y}_{j,\pmb{\theta}}(u)|<\frac{1}{c_1}$ for $j>\widehat{j}_0$, and  $ \widehat{y}_{j,\pmb{\theta}}(\bigcdot)\rightarrow o$ uniformly on $S^{n-1}$  due to   $\lim_{j\rightarrow\infty}\widehat{\lambda}_j(\pmb{\theta})=\infty$. Moreover,  $\{|\widehat{y}_{j,\pmb{\theta}}(\bigcdot)|\}_{j\in\mathbb{N}}$ is uniformly bounded on $S^{n-1}$ and hence contained in a compact set, say $\Lambda$. On $\Lambda$, $\Phi_0$ is continuous and hence is bounded from above. This, together with \eqref{def-func-conv-1}, yields  that $\Phi_j\rightarrow \Phi_0$ uniformly on $\Lambda$ and   $$\sup\Big\{\Phi_{j}(z): j\in \mathbb{N}\ \ \mathrm{and}\ \  z\in \Lambda\Big\}<\infty. $$ As $\Phi_0(o)=0$, one has,  for each $u\in S^{n-1}$, $$\lim_{j\rightarrow \infty}\Phi_{j}\big(\widehat{y}_{j,\pmb{\theta}}(u)\big)=\Phi_{0}(o)=0.$$  
It then follows from \eqref{about phi j} and the dominated convergence theorem that 
\begin{align*}
nV_n(K_0)
 =\lim_{j\rightarrow \infty}\int_{S^{n-1}}\Phi_{j}\big(\widehat{y}_{j,\pmb{\theta}}(u)\big)h_{K_0}(u)dS_{K_0}(u) =\int_{S^{n-1}}\lim_{j\rightarrow \infty}\Phi_{j}\big(\widehat{y}_{j,\pmb{\theta}}(u)\big)h_{K_0}(u)dS_{K_0}(u)=0.
\end{align*} 
This is a contradiction, and hence $\{\widehat{\lambda}_j(\pmb{\theta})\}_{j\in \mathbb{N}}$ has a finite upper bound.

We again use the argument by contradiction to prove that $\{\widehat{\lambda}_j(\pmb{\theta})\}_{j\in \mathbb{N}}$ has a positive lower bound. To this end, assume that $\lim_{j\rightarrow\infty}\widehat{\lambda}_j(\pmb{\theta})=0$. Following \eqref{M-3}, we consider   
\begin{align*}
\Lambda_{\pmb{\theta}}=\{u\in S^{n-1}: \Phi_0(a u ^{\mathrm{T}}{}_{\bigcdot}\, \pmb{\theta})=0 \ \ \mathrm{for\ some}\ \ a>0 \}.
\end{align*}
Again $\Lambda_{\pmb{\theta}}$ is contained in a closed hemisphere of $S^{n-1}$ and hence $S_{K_0}(S^{n-1}\setminus \Lambda_{\pmb{\theta}})>0.$  Moreover, $S^{n-1}\setminus \Lambda_{\pmb{\theta}}$ contains an open hemisphere. Thus, $S^{n-1}\setminus \Lambda_{\pmb{\theta}}$ can be obtained by the union of an increasing nest of its compact subsets. Due to $S_{K_0}(S^{n-1}\setminus \Lambda_{\pmb{\theta}})>0,$ one can find a compact subset, say $\Lambda_1\subseteq S^{n-1}\setminus \Lambda_{\pmb{\theta}}$, satisfying that $S_{K_0}(\Lambda_1)>0. $ Note that, for $u\in \Lambda_1$,   $\Phi_0(a u^{\mathrm{T}}{}_{\bigcdot}\,  \pmb{\theta})>0$ holds for all $a>0$, and in particular, $\Phi_0(u^{\mathrm{T}}{}_{\bigcdot}\,  \pmb{\theta})>0$. It can be checked that $ \Phi_0(u^{\mathrm{T}}{}_{\bigcdot}\,  \pmb{\theta})$ as a function on $u\in S^{n-1}$ is positive and continuous on $\Lambda_1$, and hence 
\begin{align} \label{positive-M_1}
\min_{u\in \Lambda_1} \Phi_0(u^{\mathrm{T}}{}_{\bigcdot}\,  \pmb{\theta})>0.\end{align} 
By \eqref{def-func-conv-1}, $\Phi_j\rightarrow \Phi_0$ uniformly on $B_2^m$. As $\{u^{\mathrm{T}}{}_{\bigcdot}\,  \pmb{\theta}: u\in \Lambda_1\}\subset B_2^m$, then, $\Phi_j(u^{\mathrm{T}}{}_{\bigcdot}\,  \pmb{\theta})\rightarrow \Phi_0(u^{\mathrm{T}}{}_{\bigcdot}\,  \pmb{\theta})$ uniformly on $\Lambda_1$. Without loss of generality,  due to \eqref{positive-M_1}, we can assume that 
\begin{align}\label{inf-Phi-j}
\inf\Big\{\Phi_j(u^{\mathrm{T}}{}_{\bigcdot}\,  \pmb{\theta}):  j\in \mathbb{N} \ \ \mathrm{and} \ \ u\in \Lambda_1\Big\}>0.    
\end{align}
As  $\lim_{j\rightarrow\infty}\widehat{\lambda}_j(\pmb{\theta})=0$ and $h_{K_0}>c_1$ on $S^{n-1}$, one sees that $\lim_{j\rightarrow\infty}  \widehat{\lambda}_j(\pmb{\theta})h_{K_0}(u) =0$ for each $u\in S^{n-1}.$ 
It follows from \eqref{Phi-ru-to-infty} and \eqref{inf-Phi-j}  that, for all $u\in \Lambda_1,$ 
$$\lim_{j\rightarrow\infty}\Phi_j\bigg( \frac{u^{\mathrm{T}}{}_{\bigcdot}\,  \pmb{\theta}}{\widehat{\lambda}_j(\pmb{\theta}) h_{K_0}(u)}\bigg)\geq \lim_{j\rightarrow\infty} \frac{\Phi_j(u^{\mathrm{T}}{}_{\bigcdot}\,  \pmb{\theta})}{\widehat{\lambda}_j(\pmb{\theta}) h_{K_0}(u)} =\infty.  $$
 Together with Fatou's lemma and the fact that $S_{K_0}(\Lambda_1)>0$, one gets
\begin{align*}
nV_n(K_0)&=\lim_{j\rightarrow \infty}\int_{S^{n-1}}\Phi_j\bigg( \frac{u^{\mathrm{T}}{}_{\bigcdot}\,  \pmb{\theta}}{\widehat{\lambda}_j(\pmb{\theta}) h_{K_0}(u)}\bigg)h_{K_0}(u)dS_{K_0}(u)\\
&\geq \liminf_{j\rightarrow \infty}\int_{\Lambda_1}\Phi_j\bigg( \frac{u^{\mathrm{T}}{}_{\bigcdot}\,  \pmb{\theta}}{\widehat{\lambda}_j(\pmb{\theta}) h_{K_0}(u)}\bigg)h_{K_0}(u)dS_{K_0}(u)\\
&\ge\int_{\Lambda_1 }\liminf_{j\rightarrow \infty}
\Phi_j\bigg( \frac{u^{\mathrm{T}}{}_{\bigcdot}\,  \pmb{\theta}}{\widehat{\lambda}_j(\pmb{\theta}) h_{K_0}(u)}\bigg) h_{K_0}(u)dS_{K_0}(u)=\infty,
\end{align*} a contradiction. Hence,  $\{\widehat{\lambda}_j(\pmb{\theta})\}_{j\in\mathbb{N}}$ has a positive lower bound.  
 
In summary, for each $\pmb{\theta}\in S^{nm-1},$ the sequence $\{\widehat{\lambda}_j(\pmb{\theta})\}_{j\in\mathbb{N}}$ is bounded. Thus, any subsequence of $ \{\widehat{\lambda}_j(\pmb{\theta})\}_{j\in\mathbb{N}}$, say $\{\widehat{\lambda}_{j_k}(\pmb{\theta})\}_{k\in\mathbb{N}}$, must have a convergent subsequence, say $\{\widehat{\lambda}_{j_{k_l}}(\pmb{\theta})\}_{l\in\mathbb{N}}$, such that $$\lim_{l\rightarrow \infty} \widehat{\lambda}_{j_{k_l}}(\pmb{\theta})=\widehat{\lambda}(\pmb{\theta}) \in (0, \infty). $$ 
It follows from $K_0\in\mathcal{K}_{(o)}^{n,1}$, $\Phi_{j_{k_l}}\rightarrow\Phi_0$, \eqref{about phi j} with $j$ replaced by $j_{k_l}$, and  the  dominated convergence theorem  that
 \begin{align*}
nV_n(K_0)&=\lim_{l\rightarrow\infty}\int_{S^{n-1}}\Phi_{j_{k_l}}\bigg(\frac{u^{\mathrm{T}}{}_{\bigcdot}\,  \pmb{\theta}}{\widehat{\lambda}_{j_{k_l}}(\pmb{\theta}) h_{K_0}(u)}\bigg)h_{K_0}(u)dS_{K_0}(u)\\
&=\int_{S^{n-1}}\Phi_{0}\bigg(\frac{u^{\mathrm{T}}{}_{\bigcdot}\,  \pmb{\theta}}{\widehat{\lambda}(\pmb{\theta}) h_{K_0}(u)}\bigg)h_{K_0}(u)dS_{K_0}(u).
 \end{align*}
Thus, by \eqref{lemma equiv of pro body}, for each $\pmb{\theta}\in S^{nm-1}$, one has $\widehat{\lambda}(\pmb{\theta})=h_{\Pi_{\Phi_0}^{m}K_0}(\pmb{\theta})$,  and hence, 
\begin{align*}
\lim_{j\rightarrow\infty} h_{\Pi_{\Phi_{j}}^{m}{K_0}}(\pmb{\theta})= \lim_{j\rightarrow\infty}\widehat{\lambda}_{j}(\pmb{\theta}) =h_{\Pi_{\Phi_0}^{m}K_0}(\pmb{\theta}),
\end{align*} due to the arbitrariness of $\widehat{\lambda}_{j_{k}}(\pmb{\theta}).$
This further yields the uniform convergence of   $h_{\Pi_{\Phi_{j}}^{m}{K_0}}\rightarrow h_{\Pi_{\Phi}^{m}K_0}$  on $S^{nm-1}$, and hence $\Pi_{\Phi_{j}}^{m}{K_0}\rightarrow \Pi_{\Phi}^{m}{K_0}$ as desired.
\end{proof}

For $A\in SL(n)$ and $\pmb{y}=(y_i)_i\in M_{n,m}(\mathbb{R})$ with $y_i\in M_{n,1}(\mathbb{R})$, let $A{}_{\bigcdot}\, \pmb{y}=(A{}_{\bigcdot}\, y_1,\cdots,A{}_{\bigcdot}\, y_m). $  We also let $$ A^{-1}{}_{\bigcdot}\, \pmb{y}=(A^{-1}{}_{\bigcdot}\, y_1,\cdots,A^{-1}{}_{\bigcdot}\, y_m)\ \ \mathrm{and} \ \ A^{\mathrm{-T}}{}_{\bigcdot}\, \pmb{y}=(A^{\mathrm{-T}}{}_{\bigcdot}\, y_1,\cdots,A^{\mathrm{-T}}{}_{\bigcdot}\, y_m).$$   

\begin{proposition}\label{affine inveri}
Let $K\in\mathcal{K}_{(o)}^{n,1}$ and $A\in SL(n)$. Then
\begin{align}\label{affine inv}
\Pi_{\Phi}^{m}A{}_{\bigcdot}\, K={A}^{\mathrm{-T}}{}_{\bigcdot}\, \Pi_{\Phi}^{m}K.
\end{align}
\end{proposition} 

\begin{proof} Let $u\in S^{n-1}$ be a unit outer normal vector of   $y\in \partial{K}$.  For $A\in SL(n)$,    $A{}_{\bigcdot}\, y$ is a boundary point of $A{}_{\bigcdot}\, K$ and 
\begin{align}\label{widetilde-u}
\widetilde{u}= \frac{A^{\mathrm{-T}}{}_{\bigcdot}\, u}{|A^{\mathrm{-T}}{}_{\bigcdot}\, u|} 
\end{align}
is a unit outer normal vector of $A{}_{\bigcdot}\, K$ at $A{}_{\bigcdot}\, y$. It follows from \eqref{hom-1-1}, \eqref{h AK} and \eqref{widetilde-u} that
\begin{align}\label{hAK-widetilde-u}
h_{A{}_{\bigcdot}\, K}(\widetilde{u})=h_K(A^{\mathrm{T}}{}_{\bigcdot}\, \widetilde{u})=h_K\bigg(\frac{u}{|A^{-{\mathrm{T}}}{}_{\bigcdot}\, u|}\bigg)=
\frac{h_K(u)}{|A^{-{\mathrm{T}}}{}_{\bigcdot}\, u|}.
\end{align} 
Denote $V_K$ the probability measure given  by 
\begin{align*}
d V_K(u)=\frac{h_K(u)d S_K(u)}{n V_n(K)}.
\end{align*}  
Note that
 $d V_{A{}_{\bigcdot}\, K}(\widetilde{u})=d V_K(u)$ for $A\in SL(n)$ (see \cite[P. 502]{schneider}).  By  Definition \ref{supp-pro}, \eqref{h AK}, \eqref{widetilde-u} and \eqref{hAK-widetilde-u}, one has, for any $\pmb{x}\in M_{n,m}(\mathbb{R})$,  
\begin{align*}
h_{\Pi_{\Phi}^{m}A{}_{\bigcdot}\, K}(\pmb{x})&=\inf\left\{\lambda>0:\int_{S^{n-1}}\Phi\bigg(\frac{\widetilde{u}^{\mathrm{T}}{}_{\bigcdot}\, \pmb{x}}{\lambda h_{A{}_{\bigcdot}\, K}(\widetilde{u})}\bigg)d V_{A{}_{\bigcdot}\, K}(\widetilde{u})\leq 1\right\}\\
&=\inf\left\{\lambda>0:\int_{S^{n-1}}\Phi\bigg(\frac{{(A^{-{\mathrm{T}}}{}_{\bigcdot}\, u)}^{\mathrm{T}}{}_{\bigcdot}\,  \pmb{x}}{\lambda h_{K}(u)}\bigg) d V_{K}(u)\leq 1\right\}\\
&=\inf\left\{\lambda>0:\int_{S^{n-1}}\Phi\bigg(\frac{u^{\mathrm{T}}{}_{\bigcdot}\,  (A^{-1}{}_{\bigcdot}\, \pmb{x})}{\lambda h_{K}(u)}\bigg)  d V_{K}(u)\leq 1\right\}\\
&=h_{\Pi_{\Phi}^{m}K}(A^{-1}{}_{\bigcdot}\, \pmb{x})=h_{A^{-\mathrm{T}}{}_{\bigcdot}\, \Pi_{\Phi}^{m}K}(\pmb{x}).
\end{align*} This completes the proof. 
\end{proof}

\section{Higher-order Orlicz-Petty projection inequality}\label{sec-inequality}
In this section, we will establish the affine isoperimetric inequality regarding the $m$th order Orlicz projection body. Our approach is based on the classical Steiner symmetrization (see, e.g., \cite{schneider}) and the Fiber symmetrization by McMullen \cite{Mcmullen-1999}. See Bianchi, Gardner and Gronchi \cite{Bianchi-Gardner-Gronchi-2017} and Ulivelli \cite{Ulivelli} for more details on the  Fiber symmetrization and its extensions. We begin by reviewing the relevant definitions and some properties that we need to use. 

Recall that for $v\in S^{n-1}$, $v^{\perp}$ is the orthogonal complement of $v$, i.e., $$v^{\perp}=\big\{x\in M_{n,1}(\mathbb{R}): x\bigcdot v=0\big\}.$$ Denote by  $P_{v^{\perp}}K$ the orthogonal projection of $K\in\mathcal{K}^{n,1}$ onto $v^{\perp}$. Let $f_j:P_{v^{\perp}}K\rightarrow\mathbb{R}$, $j=1,2,$ be concave functions such that
\begin{align*}
K=\Big\{z+\tau v\in M_{n,1}(\mathbb{R}): z\in P_{v^{\perp}}K\ \ \mathrm{and}\ \ -f_1(z)\leq \tau\leq f_2(z)\Big\}.
\end{align*}
If $f_1$ is differentiable at $z$, then  $\nabla f_1(z)\perp v$, where $\nabla f_1(z)$ denotes the gradient of $f_1$ at $z$, and 
the outer unit normal vector of $z-f_1(z)v \in\partial K$ is:
\begin{align}\label{nk-f1}
\nu_K\big(z-f_1(z)v\big)=\frac{-\nabla f_1(z)-v}{|-\nabla f_1(z)-v|}=\frac{-\nabla f_1(z)-v}{\sqrt{1+|\nabla f_1(z)|^2}}. 
\end{align}   Similarly, if $f_2$ is differentiable at $z$, then $\nabla f_2(z)\perp v$ and  
\begin{align}\label{nk-f2}
\nu_K\big(z+f_2(z)v\big)=\frac{-\nabla f_2(z)+v}{|-\nabla f_2(z)+v|}=\frac{-\nabla f_2(z)+v}{\sqrt{1+|\nabla f_2(z)|^2}}.
\end{align} 
Moreover, for $j=1, 2$, if we let $\langle f_j\rangle(z)=-{\nabla f_j(z)}^{\mathrm{T}}{}_{\bigcdot}\,  z+f_j(z),$ then 
\begin{align}\label{nk-f3}
 \big(z+(-1)^jf_j( z)v\big)\bigcdot\nu_K\big(z+(-1)^j f_j(z)v\big)=\frac{\langle f_j\rangle (z)}{\sqrt{1+|\nabla f_j(z)| ^2}}.   
\end{align}	
Note that $\langle \bigcdot\rangle$ is a linear operator and the concave functions $f_j$, $j=1, 2$,  are differentiable $\mathcal{H}^{n-1}$-almost everywhere on $P_{v^{\perp}}K$. If $K\in \mathcal{K}^{n,1}_{(o)}$, the origin is in the interior of $K$ and the support function of $K$ must be strictly positive on $\sphere$. This yields that both $\langle f_1\rangle$ and $\langle f_2\rangle$ are strictly positive. Define $S_vK$, the Steiner symmetrization of $K\in\mathcal{K}^{n,1}$ with respect to $v\in S^{n-1}$, by 
\begin{align}\label{stein-sym-1}
S_vK=\Big\{z+\tau v\in M_{n,1}(\mathbb{R}): z\in P_{v^{\perp}}K \ \ \mathrm{and} \ \   |\tau|\leq \frac{f_1(z)+f_2(z)}{2}\Big\}.
\end{align}  Clearly  $S_{v_j}K \rightarrow S_vK$ if $v_j\rightarrow v\in S^{n-1}$.  It is well-known that, for all $v\in \sphere,$ \begin{align}
    V_n(S_vK)=V_n(K). \label{volume-pres}
\end{align}

Let $K\in \mathcal{K}^{n,1}$ and $v\in S^{n-1}$. Denote by $\Xi_{v,K}$  the set of  points $x\in \partial K$ such that $x+\mathbb{R}v$ is tangent to $K$.
If $\mathcal{H}^{n-1}(\Xi_{v,K})=0$, one has, for any continuous function $g:\partial K\rightarrow\mathbb{R}$, 
\begin{align}\label{partial K to projection formula}
\int_{\partial K}\!\!g(w)d\mathcal{H}^{n-1}(w)
\!=\!\!\int_{P_{v^{\perp}}K}\!\!\bigg(\!g\big(z\!-\!f_1(z)v\big)\sqrt{1\!+\!|\nabla f_1(z)|^2} +g\big(z\!+\!f_2(z)v\big)\sqrt{1\!+\!|\nabla f_2(z)|^2}\!\bigg) dz.
\end{align} 

Denote $V^m(v)=V^{\perp}\times \cdots \times V^{\perp}$  the orthogonal complement of the $m$-dimensional space $[v]$ of $M_{n,m}(\mathbb{R})$, where  $[v]=\{v{}_{\bigcdot}\, \!w: w\!\in\! M_{1,m}(\mathbb{R})\}.$ 
For $\pmb{K}\in\mathcal{K}^{n,m}$, the Fiber symmetrization of $\pmb{K}$ with respect to  $v\in S^{n-1}$ is defined by
\begin{align}
\bar{S}_v \pmb{K}\!=\!\Big\{{\pmb{x}}\!+\!v{}_{\bigcdot}\, \frac{t\!+\!s}{2}: {\pmb{x}}\in \! V^m(v),\!\ s\in\! M_{1,m}(\mathbb{R}),\ t\in \! M_{1,m}(\mathbb{R}),\ {\pmb{x}}-v{}_{\bigcdot}\, s\in \pmb{K}\ \mathrm{and} \ {\pmb{x}}+v{}_{\bigcdot}\, t\in\! \pmb{K}\Big\}. \label{m-th-stein-1}
\end{align} 
Observe that, if ${\pmb{x}}-v{}_{\bigcdot}\, s\in  \partial \pmb{K}$ and ${\pmb{x}}+v{}_{\bigcdot}\, t\in \partial \pmb{K}$ with  $-s \ne t$, then $${\pmb{x}}+\frac{v{}_{\bigcdot}\, {(t+s)}}{2}\in \partial\bar{S}_v \pmb{K}.$$ 
Clearly, $\bar{S}_{v_j}\pmb{K}\rightarrow \bar{S}_v\pmb{K}$ if $v_j\rightarrow v\in S^{n-1}$.
It follows from \cite[Lemma 3.1]{Ulivelli}  that,
\begin{align}\label{vnm l}
V_{nm}(\pmb{K})\leq V_{nm}(\bar{S}_v \pmb{K})\ \ \mathrm{for}\ \pmb{K}\in\mathcal{K}^{n,m} \ \ \mathrm{and}
\ \ v\in S^{n-1}.
\end{align} 

We shall need the following proposition, which  provides the relationship between $\bar{S}_{v}\Pi_{\Phi}^{m,*}K$ and $\Pi_{\Phi}^{m,*}S_vK$ for $K\in \mathcal{K}^{n,1}_{(o)}$. Recall that $\Pi_{\Phi}^{m,*}K=(\Pi_{\Phi}^{m})^*K$ is the polar body of $\Pi_{\Phi}^{m}K.$
\begin{proposition}\label{key}
Let $\Phi\in\mathcal{C}$ and $K\in\mathcal{K}_{(o)}^{n,1}$. For any $v\in S^{n-1}$, one has
\begin{align}\label{inequality}
\bar{S}_{v}\Pi_{\Phi}^{m,*}K\subseteq \Pi_{\Phi}^{m,*}S_vK.
\end{align}	
If in addition $\Phi\in\mathcal{C}$ is strictly convex and $\bar{S}_{v}\Pi_{\Phi}^{m,*}K=\Pi_{\Phi}^{m,*}S_vK$,  then all the midpoints of the chords of $K$ parallel to $v$ lie in a subspace of $M_{n,1}(\mathbb{R}).$
\end{proposition}
\begin{proof} It follows from the Ewald-Larman-Rogers theorem \cite[Theorem 4.3]{Ewald Larman Rogers} that $\mathcal{H}^{n-1}(\Xi_{v,K})=0$ holds for almost all $v\in S^{n-1}$. Due to the continuity of $\bar{S}_v$ with respect to $v$ and the continuity of $\Pi_{\Phi}^{m,*}:\mathcal{K}^{n,1}_{(o)}\rightarrow\mathcal{K}_{(o)}^{n,m}$, it is enough to prove \eqref{inequality} for $v\in \sphere$ such that  $\mathcal{H}^{n-1}(\Xi_{v,K})=0$.

To this end, let $\pmb{\xi}+v{}_{\bigcdot}\, r=({\xi}_i+r_iv)_i\in \bar{S}_{v}\Pi_{\Phi}^{m,*}K$ with $\pmb{\xi}\in V^m(v)$ and $r\in M_{1,m}(\mathbb{R})$, where $$\pmb{\xi}=({\xi}_i)_i=(\xi_1, \cdots, \xi_m) \ \ \mathrm{and}\ \  r=(r_1,\cdots,r_m)$$ with $\xi_i\in M_{n, 1}(\R)$ and $r_i\in\mathbb{R}$ for $i=1,\cdots\!,m$. By the definition of $\bar{S}_{v}$, there exist $s\!=\!(s_1,\cdots, s_m)\in M_{1,m}(\mathbb{R})$ and $t\!=\!(t_1,\cdots, t_m)\!\in\!M_{1,m}(\mathbb{R})$ such that \begin{align}
    \pmb{\xi}-v{}_{\bigcdot}\, s\in\Pi_{\Phi}^{m,*}K,\ \  \pmb{\xi}+v{}_{\bigcdot}\,  t \in\Pi_{\Phi}^{m,*}K\ \ \mathrm{and}\ \ r=\frac{ t+s } {2}. \label{radial<1}
\end{align} 

We now claim that $h_{\Pi_{\Phi}^mS_vK}(\pmb{\xi}+v{}_{\bigcdot}\,  r)\leq 1$,  which  yields  $\pmb{\xi}+v{}_{\bigcdot}\,  r\in\Pi_{\Phi}^{m,*}S_vK$, and hence formula \eqref{inequality} holds. It follows from \eqref{radial<1} that \begin{align}
   h_{\Pi_{\Phi}^mK}(\pmb{\xi}-v{}_{\bigcdot}\, s)\leq 1 \ \ \mathrm{and} \ \ h_{\Pi_{\Phi}^mK}(\pmb{\xi}+v{}_{\bigcdot}\, t)\leq 1. \label{radial-<1} 
\end{align} Recall the fact showed in the proof of  Proposition \ref{Proposition-Hx} that, for any $\pmb{x}\in M_{n, m}(\R)$, $H_{\pmb{x}}(0)=0$ and $H_{\pmb{x}}(t)$ is nonnegative and convex on $t\in (0, \infty)$. Then $H_{\pmb{x}}(\bigcdot)$ is increasing on $(0, \infty).$  Together with \eqref{ S k and H n-1}, \eqref{lemma equiv of pro body}, \eqref{nk-f1}, \eqref{nk-f2}, \eqref{nk-f3}, \eqref{partial K to projection formula}, \eqref{radial-<1}, and $\mathcal{H}^{n-1}(\Xi_{v,K})=0$, one has  
\begin{align}\label{1.2} nV_n(K)&=H_{\pmb{\xi}-v{}_{\bigcdot}\,  s}\Big(\frac{1}{h_{\Pi_{\Phi}^mK}(\pmb{\xi}-v{}_{\bigcdot}\, s)}\Big)\nonumber\\ & \geq H_{\pmb{\xi}-v{}_{\bigcdot}\,  s}(1) \nonumber \\  & =\int_{S^{n-1}}\Phi\bigg(\frac{u^{\mathrm{T}}{}_{\bigcdot}\,  (\pmb{\xi}-v{}_{\bigcdot}\,  s)}{h_K(u)}\bigg)h_K(u)dS_K(u)\nonumber\\
& =\int_{\partial K}\Phi\bigg(\frac{{\nu_K(w)}^{\mathrm{T}}{}_{\bigcdot}\,  (\pmb{\xi}-v{}_{\bigcdot}\,  s)}{w\bigcdot \nu_K(w)}\bigg)w\bigcdot \nu_K(w)d\mathcal{H}^{n-1}(w)\nonumber\\
&= \int_{P_{v^{\perp}}K} \bigg[\Phi\bigg(\frac{-\nabla f_1(z)^{\mathrm{T}}{}_{\bigcdot}\, \pmb{\xi}+s}{\langle f_1\rangle(z)}\bigg)\langle f_1\rangle(z) +\Phi\bigg(\frac{-\nabla f_2(z)^{\mathrm{T}}{}_{\bigcdot}\, \pmb{\xi}-s}{\langle f_2\rangle(z)}\bigg)\langle f_2\rangle(z)\bigg] dz.
\end{align} Along the same lines, one gets  
\begin{align}\label{1.1} nV_n(K)\geq  \int_{P_{v^{\perp}}K} \bigg[\Phi\bigg(\frac{-\nabla f_1(z)^{\mathrm{T}}{}_{\bigcdot}\, \pmb{\xi}-t}{\langle f_1\rangle(z)}\bigg)\langle f_1\rangle(z) +\Phi\bigg(\frac{-\nabla f_2(z)^{\mathrm{T}}{}_{\bigcdot}\, \pmb{\xi}+t}{\langle f_2\rangle(z)}\bigg)\langle f_2\rangle(z)\bigg]dz.
\end{align}
It also can be checked from \eqref{nk-f1}, \eqref{nk-f2}, \eqref{nk-f3}, \eqref{stein-sym-1} and \eqref{partial K to projection formula} that 
\begin{align*} 
\!\!\int_{S^{n-1}}\!\!\!\! \Phi\bigg(\!\frac{u^{\mathrm{T}}{}_{\bigcdot}\,  (\pmb{\xi}+v{}_{\bigcdot}\,  r)}{h_{S_vK}(u)}\!\bigg)h_{S_vK}(u)dS_{S_vK}(u) =&\!\int_{P_{v^{\perp}}K}\!\!\!\! \Phi\bigg(\!\frac{-\nabla(f_1\!+\!f_2)(z)^{\mathrm{T}}{}_{\bigcdot}\, \pmb{\xi}-(s+t)}{\langle f_1+f_2\rangle(z)}\!\bigg)\frac{\langle f_1\!+\!f_2\rangle(z)}{2} dz \nonumber \\
&+\!\int_{P_{v^{\perp}}K}\!\!\!\! \Phi\bigg(\frac{-\nabla(f_1\!+\!f_2)(z)^{\mathrm{T}}{}_{\bigcdot}\, \pmb{\xi}+(s+t)}{\langle f_1\!+\! f_2\rangle(z)}\bigg)\frac{\langle f_1\!+\!f_2\rangle(z)}{2}  dz.
\end{align*} 
Let $H_{\pmb{\xi}+v{}_{\bigcdot}\,  r}(\bigcdot)$ be given in \eqref{H-x} with $K$ replaced by $S_v K$.
Note that $\langle f_1+f_2\rangle\!=\!\langle f_1 \rangle+\langle f_2\rangle$. Combining with the convexity of $\Phi$ and the equality above, one gets
\begin{align*}
H_{\pmb{\xi}+v{}_{\bigcdot}\,  r}(1)=&\int_{S^{n-1}}\!\!\!\! \Phi\bigg(\!\frac{u^{\mathrm{T}}{}_{\bigcdot}\,  (\pmb{\xi}+v{}_{\bigcdot}\,  r)}{h_{S_vK}(u)}\!\bigg)h_{S_vK}(u)dS_{S_vK}(u)\\ 
\leq&\frac{1}{2}\int_{P_{v^{\perp}}K}\Phi\bigg(\frac{-\nabla f_1(z)^{\mathrm{T}}{}_{\bigcdot}\, \pmb{\xi}-t}{\langle f_1\rangle(z)}\bigg)\langle f_1\rangle(z) dz+\frac{1}{2}\int_{P_{v^{\perp}}K}\Phi\bigg(\frac{-\nabla f_2(z)^{\mathrm{T}}{}_{\bigcdot}\, \pmb{\xi}-s}{\langle f_2\rangle(z)}\bigg)\langle f_2\rangle(z) dz\\
&+\frac{1}{2}\int_{P_{v^{\perp}}K}\Phi\bigg(\frac{-\nabla f_1(z)^{\mathrm{T}}{}_{\bigcdot}\, \pmb{\xi}+s}{\langle f_1\rangle(z)}\bigg)\langle f_1\rangle(z)dz+\frac{1}{2}\int_{P_{v^{\perp}}K}\Phi\bigg(\frac{-\nabla f_2(z)^{\mathrm{T}}{}_{\bigcdot}\, \pmb{\xi}+t}{\langle f_2\rangle(z)}\bigg)\langle f_2\rangle(z)dz.
\end{align*}
Then, it follows from \eqref{1.2} and \eqref{1.1} that 
\begin{align}\label{equ.}
H_{\pmb{\xi}+v{}_{\bigcdot}\,  r}(1)=
\int_{S^{n-1}}\Phi\bigg(\frac{u^{\mathrm{T}}{}_{\bigcdot}\,  (\pmb{\xi}+v{}_{\bigcdot}\,  r)}{h_{S_vK}(u)}\bigg)h_{S_vK}(u)dS_{S_vK}(u) \leq nV_n(K). 
\end{align} 
Together with the monotonicity of $H_{\pmb{x}}(t)$ on $t\in (0,\infty)$, one gets that  $h_{\Pi_{\Phi}^mS_vK}(\pmb{\xi}+v{}_{\bigcdot}\,  r)\!\leq \!1$, and hence $\pmb{\xi}+v{}_{\bigcdot}\, r\in\Pi_{\Phi}^{m,*}S_vK$. This completes the proof of \eqref{inequality}.

Now let us characterize the equality. Suppose that $\Phi\in\mathcal{C}$ is strictly convex and $\bar{S}_{v}\Pi_{\Phi}^{m,*}K=\Pi_{\Phi}^{m,*}S_vK$. For $\pmb{\xi}\in V^m(v)$, let $\pmb{\xi}-v{}_{\bigcdot}\, s\in \partial\Pi_{\Phi}^{m,*}K$ and $\pmb{\xi}+v{}_{\bigcdot}\, t\in \partial\Pi_{\Phi}^{m,*}K$ with $-s\ne t$.  Thus $$\pmb{\xi}+v{}_{\bigcdot}\,  r\in\partial\bar{S_v}\Pi_{\Phi}^{m,*}K \ \ \mathrm{where} \ \ r=\frac{ t+s } {2}.$$  Consequently, due to $\bar{S}_{v}\Pi_{\Phi}^{m,*}K=\Pi_{\Phi}^{m,*}S_vK$, one has 
\begin{align*}
h_{\Pi_{\Phi}^mK}(\pmb{\xi}-v{}_{\bigcdot}\, s)=h_{\Pi_{\Phi}^mK}(\pmb{\xi}+v{}_{\bigcdot}\, t)=h_{\bar{S}_v\Pi_{\Phi}^mK}(\pmb{\xi}+v{}_{\bigcdot}\,  r)=h_{\Pi_{\Phi}^mS_vK}(\pmb{\xi}+v{}_{\bigcdot}\,  r)=1. 
\end{align*} 
These in particular imply equality in \eqref{equ.}, and hence equalities hold  in \eqref{1.2} and  \eqref{1.1}. Since $\Phi\in\mathcal{C}$ is strictly convex,  
\begin{align}\label{equality-case-1}
  \frac{-\nabla f_1(z)^{\mathrm{T}}{}_{\bigcdot}\, \pmb{\xi}-t}{\langle f_1\rangle(z)}=\frac{-\nabla f_2(z)^{\mathrm{T}}{}_{\bigcdot}\, \pmb{\xi}-s}{\langle f_2\rangle(z)}\ \ \ \mathrm{and} \ \ \ \frac{-\nabla f_1(z)^{\mathrm{T}}{}_{\bigcdot}\, \pmb{\xi}+s}{\langle f_1\rangle(z)}=\frac{-\nabla f_2(z)^{\mathrm{T}}{}_{\bigcdot}\, \pmb{\xi}+t}{\langle f_2\rangle(z)} 
\end{align} hold for almost all $z\in P_{v^{\perp}}K$.  Subtracting the first equation in \eqref{equality-case-1} by the second equation in \eqref{equality-case-1}, one gets, for almost all $z\in P_{v^{\perp}}K$  
$$\frac{s+t}{\langle f_1\rangle(z)}=\frac{s+t}{\langle f_2\rangle(z)}.$$ As $s\neq -t$, one must have,   for almost all $z
\in P_{v^{\perp}}K$,  \begin{align}\label{equality-case-2}
\langle f_1\rangle(z)=\langle f_2\rangle(z),\ \ 
\mathrm{i.e.,}\ \ (f_1-f_2)(z)=\nabla(f_1-f_2)(z)^{\mathrm{T}}{}_{\bigcdot}\, z.
\end{align} 
Combining with \eqref{equality-case-1}, one gets
$s-t=\nabla (f_1-f_2)(z)^{\mathrm{T}}{}_{\bigcdot}\, \pmb{\xi}$ holds for almost all $z
\in P_{v^{\perp}}K$, 
which further yields that, for any $i=1,\cdots, m$, 
\begin{align}\label{nabla-f1-f2}
s_i-t_i=\nabla (f_1-f_2)(z)\bigcdot \xi_i. 
\end{align}

Note that $z$ and $\pmb{\xi}$ are independent of each other. Thus, for almost all $z
\in P_{v^{\perp}}K$,  if $\xi_i\neq o$, \eqref{nabla-f1-f2} asserts that $\nabla (f_1-f_2)(z)$ lies in a hyperplane with normal direction $\frac{\xi_i}{|\xi_i|}$. As $\pmb{\xi}\in V^{m}(v)$ is arbitrary, one can indeed get $\nabla (f_1-f_2)(z)=\widehat{w}$ for almost all $z\in P_{v^{\perp}}K$, where $\widehat{w}\in M_{n,1}(\mathbb{R})$ is independent of $z.$ It follows from \eqref{equality-case-2} that, for almost all $z\in P_{v^{\perp}}K$,
\begin{align*}
f_1(z)-f_2 (z)={\widehat{w}}^{\mathrm{T}}{}_{\bigcdot}\, z,
\end{align*}
and hence, obviously it holds for all $z\in P_{v^{\perp}}K$. Thus, the midpoints of the chords of $K$ parallel to $v$ lie in the subspace:
$$\Big\{z+  \frac{{\widehat{w}}^{\mathrm{T}}{}_{\bigcdot}\, z}{2} v: z\in P_{v^{\perp}}K\Big\}.$$ 
This completes the proof.
\end{proof}

 We are now in the position to prove the  higher-order Orlicz-Petty projection inequality.  Let \begin{align*}
  \Gamma_{\Phi}(K)=\frac{V_{nm}(\Pi_{\Phi}^{m,*}K)}{V_{n}(K)^m}\ \  \mathrm{for}\ \  K\in\mathcal{K}_{(o)}^{n,1}.  
 \end{align*}
 
 \begin{theorem}\label{Main-Theory-1}
Let $\Phi\in\mathcal{C}.$ Then, among $K\in\mathcal{K}_{(o)}^{n,1}$, $\Gamma_{\Phi}(K)$ 
is maximized at origin-symmetric ellipsoids, i.e., \eqref{Main-Theory-1-ineuqlity} holds: 
\begin{align*}
\Gamma_{\Phi}(K)\leq \Gamma_{\Phi}(\ball)\ \ \mathrm{for\ all}\ \ K\in \mathcal{K}_{(o)}^{n,1}.
\end{align*}  If $\Phi$ is strictly convex, then the origin-symmetric ellipsoids are the only maximizers.
\end{theorem}

\begin{proof} 
It follows from \eqref{volume-pres}, \eqref{vnm l} and  Proposition \ref{key} that, for any $v\in \sphere$, 
\begin{align}\label{result-inequality}
\frac{V_{nm}(\Pi_{\Phi}^{m,*}K)}{V_{n}(K)^m}
\leq\frac{V_{nm}(\bar{S}_v\Pi_{\Phi}^{m,*}K)}{V_{n}(S_vK)^m}
\leq\frac{V_{nm}(\Pi_{\Phi}^{m,*}S_vK)}{V_{n}(S_vK)^m}.
\end{align} 
Let $\{v_j\}_{j\in\mathbb{N}}$ be a sequence in $S^{n-1}$,  $K_1=S_{v_1}K$  and 
  $K_{j+1}=S_{v_{j+1}}K_j$ for $j\in \mathbb{N}$  such that   $K_j\rightarrow cB_2^n$ with  $ c= \big(\frac{V_n(K)}{V_n(B^n_2)}\big)^{\frac{1}{n}}$. It follows from \eqref{result-inequality}  that      
\begin{align}\label{chain-inequalities}
\frac{V_{nm}(\Pi_{\Phi}^{m,*}K)}{V_{n}(K)^m}
\leq\frac{V_{nm}(\Pi_{\Phi}^{m,*}{K_1})}{V_{n}(K_1)^m}
\leq\cdots\leq\frac{V_{nm}(\Pi_{\Phi}^{m,*}{K_j})}{V_{n}(K_j)^m}.
\end{align} By Proposition \ref{continuity}, one has  $$ \lim_{j\rightarrow\infty} \frac{V_{nm}(\Pi_{\Phi}^{m,*}{K_j})}{V_{n}(K_j)^m}=\frac{V_{nm}(\Pi_{\Phi}^{m,*}cB^n_2)}{V_{n}(cB^n_2)^m}=\frac{V_{nm}(\Pi_{\Phi}^{m,*}B^n_2)}{V_{n}(B^n_2)^m},$$    
where the last equality follows from    $\Pi_{\Phi}^{m,*}cK=c\Pi_{\Phi}^{m,*}K$, an easy consequence of   \eqref{hom-1-1}, Definition \ref{supp-pro}  and the fact that $S_{cK}=c^{n-1}S_K$ for $c>0$. Together with \eqref{chain-inequalities}, one gets the following inequality: 
\begin{align}\label{chain-inequalities-1}
\frac{V_{nm}(\Pi_{\Phi}^{m,*}K)}{V_{n}(K)^m}
\leq \frac{V_{nm}(\Pi_{\Phi}^{m,*}B^n_2)}{V_{n}(B^n_2)^m}.\end{align} By \eqref{affine inv}, for any $A\in SL(n)$,  \begin{align*} 
V_{nm}(\Pi_{\Phi}^{m}A{}_{\bigcdot}\, K) = V_{nm}(\Pi_{\Phi}^{m}K).
\end{align*} Combining with  \eqref{chain-inequalities-1}, one gets  $\Gamma_{\Phi}(K)$ is maximized when $K$ is an origin-symmetric ellipsoid.

We now characterize the equality.  
Suppose that $\Phi\in\mathcal{C}$ is strictly convex and equality holds in \eqref{chain-inequalities-1}. Then equality must hold in \eqref{result-inequality} for any $v\in \sphere$, implying that  $$V_{nm}(\bar{S}_{v}\Pi_{\Phi}^{m,*}K)=V_{nm}(\Pi_{\Phi}^{m,*}S_vK).$$ Thus, $ \bar{S}_{v}\Pi_{\Phi}^{m,*}K = \Pi_{\Phi}^{m,*}S_vK $ holds for all $v\in \sphere.$ According to  
Proposition \ref{key},  all the midpoints of the chords of $K$ parallel to $v$ lie in a subspace of $M_{n,1}(\mathbb{R})$, and thus, $K$ is an origin-symmetric ellipsoid (see e.g., \cite[Theorem 10.2.1]{schneider}). This completes the proof.
\end{proof}

\section{Higher-order Orlicz-Petty projection inequality when $\Phi\!=\!\phi\!\circ\!h_Q$}\label{sec-characterization-equality}
In Theorem \ref{Main-Theory-1}, the equality case of the higher-order Orlicz-Petty projection inequality is characterized when $\Phi\in \mathcal{C}$ is strictly convex. However, many commonly used convex functions are not necessarily strictly convex, for example, $\Phi_Q=\phi(h_Q)$ where $Q\in\mathcal{K}_{o}^{1,m}$ and  $\phi:[0, \infty)\rightarrow [0,\infty)$ is a convex function such that $\phi(0)=0$ and $\phi$ is strictly increasing on $[0,\infty)$. The main goal in this section is to provide the higher-order Orlicz-Petty projection inequality for $\Phi_Q=\phi\circ h_Q$, and characterize the corresponding equality. The following theorem from Gruber \cite[P.73]{Gruber-1974} will be used.
\begin{theorem}\label{Gruber-theorem-3}
Let $K\in\mathcal{K}^{n,1}$. Then $K$ is an ellipsoid if and only if it satisfies that, for any family of parallel chords of $K$, there exists an $r\in(0,1)$ such that the points, which divide the (directed) chords of $K$ in the proportion $r: 1-r$, lie in a hyperplane.
\end{theorem}
As claimed in Section \ref{sec-definition}, $\Phi_Q \in \mathcal{C}.$ For $K\in\mathcal{K}_{(o)}^{n,1}$,
let $\Pi_{\phi, Q}^{m}K$ be defined through its support function $h_{\Pi_{\phi, Q}^{m}K}$ in \eqref{specical-case-1}. An immediate consequence of Theorem \ref{Main-Theory-1} is the following higher-order Orlicz-Petty projection inequality corresponding to $\Pi_{\phi, Q}^{m, *}K=(\Pi_{\phi, Q}^{m}K)^*$. For convenience, let $\Gamma_{\phi, Q}: \mathcal{K}_{(o)}^{n,1}\rightarrow [0, \infty)$ be defined by  \begin{align*}\Gamma_{\phi, Q}(K)=
\frac{V_{nm}(\Pi_{\phi, Q}^{m, *}K)}{V_{n}(K)^m}\ \ \mathrm{for}\ \ K\in \mathcal{K}_{(o)}^{n,1}.
\end{align*}   

\begin{theorem} \label{Phi-Q-strictly convex}
Let  $Q\in\mathcal{K}_{o}^{1,m}$ and  $\phi:[0, \infty)\rightarrow [0,\infty)$ be a convex function such that $\phi(0)=0$ and $\phi$ is strictly increasing on $[0,\infty)$.  Then, among $K\in \mathcal{K}_{(o)}^{n,1}$, $\Gamma_{\phi, Q}(K)$ is maximized at origin-symmetric ellipsoids, i.e., \begin{align} \label{inequality-Phi-Q}
\Gamma_{\phi, Q}(K)\leq \Gamma_{\phi, Q}(\ball)\ \ \mathrm{for\ all}\ \ \mathcal{K}_{(o)}^{n,1}.\end{align}
 
If in addition $\phi$ is strictly convex on $[0,\infty)$, then the origin-symmetric ellipsoids are the only maximizers of  $\Gamma_{\phi, Q}(K)$. 
\end{theorem}

\begin{proof} The inequality \eqref{inequality-Phi-Q} clearly holds by Theorem \ref{Main-Theory-1}.  We shall only prove the characterization of equality. To this end, let $\phi$ be strictly convex on $[0,\infty)$. Suppose that $\widetilde{K}\in \mathcal{K}_{(o)}^{n,1}$ maximizes $\Gamma_{\phi, Q}(K)$ among $K\in \mathcal{K}^{n,1}_{(o)}$. 
Similar to the proof of Proposition \ref{key}, we only consider those $v\in \sphere$ such that $\mathcal{H}^{n-1}(\Xi_{v,\widetilde{K}})=0$. Recall that the spherical measure of $\big\{ v\in \sphere: \mathcal{H}^{n-1}(\Xi_{v,\widetilde{K}})\ne 0\big\}$ is $0$. By \eqref{result-inequality}, for almost all $v\in\sphere$, we have
$$V_{nm}(\bar{S}_v\Pi_{\phi, Q}^{m, *}\widetilde{K})=V_{nm}(\Pi_{\phi, Q}^{m, *}S_v\widetilde{K}).$$ 
It follows from Proposition \ref{key} that 
\begin{align}\label{equal-6.1-1}
\bar{S}_v\Pi_{\phi, Q}^{m, *}\widetilde{K}=\Pi_{\phi, Q}^{m, *}S_v\widetilde{K}.
\end{align} It has been proved in \cite[Proposition 2.1]{Haddad-ye-lp-2025} that, for   $Q\in\mathcal{K}_o^{1,m}$, there exists a projection matrix $P=\iota^{\mathrm{T}}{}_{\bigcdot}\, \kappa$ with $\iota, \kappa\in M_{1,m}(\mathbb{R})$  such that $Q{}_{\bigcdot}\, P$ lies in a one-dimensional linear subspace  and $Q{}_{\bigcdot}\, P\subset Q$.  

Note that $Q{}_{\bigcdot}\, \iota^{\mathrm{T}} $ is the segment in $\mathbb{R}$ with endpoints $-h_Q(-\iota)$ and $h_Q(\iota)$. As $o\in Q$ and $Q$ has nonempty interior, at least one of $h_Q(\iota)$ and $h_Q(-\iota)$ must be positive. Without loss of generality, assume that $h_Q(\iota)>0$. By \eqref{hom-1-1} and \eqref{h KB}, it can be checked that the function $\phi\circ h_{Q{}_{\bigcdot}\, \iota^{\mathrm{T}}}:\mathbb{R}\rightarrow [0,\infty)$ can be formulated by  
\begin{align}
\label{positive-def-p-q}&\big(\phi\circ h_{Q{}_{\bigcdot}\, \iota^{\mathrm{T}}}\big)(\tau)=\big(\phi\circ h_{Q}\big)(\tau{}_{\bigcdot}\, \iota) =\phi\big(\tau  h_Q(\iota)\big) \ \ \mathrm{for}\ \ \tau\geq 0; \\  &\big(\phi\circ h_{Q{}_{\bigcdot}\, \iota^{\mathrm{T}}}\big)(\tau)=\big(\phi\circ h_{Q}\big)(\tau{}_{\bigcdot}\, \iota) =\phi\big(\!\!-\tau h_Q(-\iota)\big) \ \ \mathrm{for}\ \ \tau< 0 \nonumber. 
\end{align} As $\phi$ is a convex monotone increasing function, for any $\eta\in [0, 1]$, $\tau\in \R$ and $\widetilde{\tau}\in \R$, one has, \begin{align*}
  \big(\phi\circ h_{Q{}_{\bigcdot}\, \iota^{\mathrm{T}}}\big)(\eta \tau+(1-\eta)\widetilde{\tau})
  &= \phi \Big( h_{Q}\big( (\eta \tau+(1-\eta)\widetilde{\tau} ){}_{\bigcdot}\, \iota \big)\Big)\\ 
  & =\phi \Big( h_{Q}\big(  \eta \tau{}_{\bigcdot}\,  \iota+(1-\eta)\widetilde{\tau} {}_{\bigcdot}\, \iota \big)\Big)\\
  &\leq \phi \big(\eta  h_{Q}(\tau{}_{\bigcdot}\,  \iota)+(1-\eta) h_{Q}(\widetilde{\tau}{}_{\bigcdot}\,  \iota )\big)\\
  &\leq \eta \phi \big( h_{Q}(\tau{}_{\bigcdot}\,  \iota)\big)+(1-\eta) \phi \big( h_{Q}(\widetilde{\tau}{}_{\bigcdot}\,  \iota )\big)\\
  &=\eta \big(\phi\circ h_{Q{}_{\bigcdot}\, \iota^{\mathrm{T}}}\big)(\tau)+(1-\eta) \big(\phi\circ h_{Q{}_{\bigcdot}\, \iota^{\mathrm{T}}}\big)( \widetilde{\tau}).
\end{align*} That is, $ \phi\circ h_{Q{}_{\bigcdot}\, \iota^{\mathrm{T}}} $ defines a convex function on $\R.$ As  $\phi(0)=0$, it is easy to see that  $\phi\circ h_{Q{}_{\bigcdot}\, \iota^{\mathrm{T}}}(0)=0$. 
Following from the fact that $\phi$ is strictly increasing on $[0, \infty)$ and   \eqref{positive-def-p-q}, one sees that $\phi\circ h_{Q{}_{\bigcdot}\, \iota^{\mathrm{T}}}$ is strictly increasing on $[0,\infty)$. 

For convenience, let $\psi=\phi\circ h_{Q{}_{\bigcdot}\, \iota^{\mathrm{T}}}$, and for $\widetilde{K}\in \mathcal{K}^{n,1}_{(o)}$, let $\Pi_{\psi}\widetilde{K}$ be the Orlicz projection body defined in \eqref{def-orlicz-petty-body}. That is, the support function of $\Pi_{\psi}\widetilde{K}$ is given by: for $x\in M_{n,1}(\mathbb{R})$,   \begin{align}\label{def-orlicz-petty-body-psi}
h_{\Pi_{\psi} \widetilde{K}}(x)&=\inf\left\{\frac{1}{t}>0:\int_{S^{n-1}}\psi\bigg(t\frac{x\bigcdot u}{ h_{\widetilde{K}}(u)}\bigg)h_{\widetilde{K}}(u)dS_{\widetilde{K}}(u)\leq n V_n({\widetilde{K}})\right\} \nonumber \\ &=\inf\left\{\frac{1}{t}>0:\int_{S^{n-1}} \phi\circ h_{Q{}_{\bigcdot}\, \iota^{\mathrm{T}}}\bigg(t\frac{x\bigcdot u}{ h_{\widetilde{K}}(u)}\bigg)h_{\widetilde{K}}(u)dS_{\widetilde{K}}(u)\leq n V_n({\widetilde{K}})\right\}. 
\end{align} 
Together with \eqref{h KB}, \eqref{specical-case-1} and $u\bigcdot x=u^{\mathrm{T}}{}_{\bigcdot}\, x$ for $u,x\in M_{n,1}(\mathbb{R})$, one gets, for $x\in M_{n,1}(\mathbb{R})$,
\begin{align*}
h_{\Pi_{\psi} \widetilde{K}}(x)
&=\inf\left\{\frac{1}{t}>0: \int_{S^{n-1}} \phi\circ h_{Q{}_{\bigcdot}\, \iota^{\mathrm{T}}} \Big(t \frac{u \bigcdot x}{h_{\widetilde{K}}({u})}\Big)h_{\widetilde{K}}({u})dS_{\widetilde{K}}({u})\leq nV_n({\widetilde{K}})\right\}\\
&=\inf\left\{\frac{1}{t}>0: \int_{S^{n-1}} \phi\circ h_Q \Big(t \frac{u^{\mathrm{T}}{}_{\bigcdot}\,  x{}_{\bigcdot}\, \iota}{h_{\widetilde{K}}({u})}\Big)h_{\widetilde{K}}({u})dS_{\widetilde{K}}({u})\leq nV_n({\widetilde{K}})\right\}\\
&=h_{(\Pi_{\phi, Q}^{m}{\widetilde{K}})}(x{}_{\bigcdot}\, \iota)=h_{(\Pi_{\phi, Q}^{m}{\widetilde{K}}){}_{\bigcdot}\, \iota^{\mathrm{T}}}(x). 
\end{align*} This implies that  
\begin{align}\label{formula-6-1}
\Pi_{\psi}{\widetilde{K}}=(\Pi_{\phi, Q}^m{\widetilde{K}}){}_{\bigcdot}\, \iota^{\mathrm{T}}.
\end{align} Again, for convenience, let $\Pi_{\psi}^*{\widetilde{K}}=(\Pi_{\psi}{\widetilde{K}})^*.$ 
Define $J_{\iota}: M_{n,1}(\mathbb{R})\rightarrow M_{n,m}(\mathbb{R})$ to be  $J_{\iota}(z)=z{}_{\bigcdot}\, \iota$. The preimage of $J_{\iota}$ is denoted by $J_{\iota}^{-1}$. We now show that
\begin{align}\label{equal-6.1-4}
\Pi_{\psi}^{*}S_v{\widetilde{K}}=J_{\iota}^{-1}(\Pi_{\phi, Q}^{m, *}S_v{\widetilde{K}}).
\end{align}
Let $y\in\Pi_{\psi}^{*}S_v{\widetilde{K}},$ i.e., 
$h_{\Pi_{\psi}S_v{\widetilde{K}}}(y)\leq
1.$ Together with \eqref{h KB} and \eqref{formula-6-1},  one has
\begin{align*}
h_{\Pi_{\psi}S_v{\widetilde{K}}}(y)=
h_{(\Pi_{\phi, Q}^{m}S_v{\widetilde{K}}){}_{\bigcdot}\, \iota^{\mathrm{T}}}(y)=h_{\Pi_{\phi, Q}^{m}S_v{\widetilde{K}}}(y{}_{\bigcdot}\, \iota)\leq1,
\end{align*} which further gives $y{}_{\bigcdot}\, \iota\in \Pi_{\phi, Q}^{m,*}S_v{\widetilde{K}}$, and hence $y\in J_{\iota}^{-1}(\Pi_{\phi, Q}^{m, *}S_v{\widetilde{K}})$. That is, we have proved that \begin{align*}
\Pi_{\psi}^{*}S_v{\widetilde{K}}\subset J_{\iota}^{-1}(\Pi_{\phi, Q}^{m, *}S_v{\widetilde{K}}).
\end{align*} The above argument can be reversed to obtain $J_{\iota}^{-1}(\Pi_{\phi, Q}^{m, *}S_v{\widetilde{K}})\subset \Pi_{\psi}^{*}S_v{\widetilde{K}},$  which yields \eqref{equal-6.1-4}. Together with \eqref{equal-6.1-1}, we obtain
\begin{align*}
\Pi_{\psi}^{*}S_v{\widetilde{K}}=J_{\iota}^{-1}(\Pi_{\phi, Q}^{m, *}S_v{\widetilde{K}})=J_{\iota}^{-1}(\bar{S}_v\Pi_{\phi, Q}^{m, *}{\widetilde{K}}).
\end{align*}

Next, we show that, for all $v\in \sphere,$ 
\begin{align}\label{equal-6.1-2}
S_v\Pi_{\psi}^{*}\widetilde{K}=\Pi_{\psi}^{*}S_v\widetilde{K}. 
\end{align} It has been proved in  \cite[Corollary 3.1]{orlicz projection} that $
 S_v\Pi_{\psi}^{*}\widetilde{K} \subset\Pi_{\psi}^{*}S_v\widetilde{K}.$ To get \eqref{equal-6.1-2}, it is enough to prove  
\begin{align}\label{equal-6.1-2.1}
\Pi_{\psi}^{*}S_v\widetilde{K}\subset S_v\Pi_{\psi}^{*}\widetilde{K} \ \ \mathrm{or\ \ equivalently}\ \ J_{\iota}^{-1}(\bar{S}_v\Pi_{\phi, Q}^{m, *}{\widetilde{K}})\subset
S_v\Pi_{\psi}^{*}{\widetilde{K}}.
\end{align} To this end, let $z+\tau v\in J_{\iota}^{-1}(\bar{S}_v\Pi_{\phi, Q}^{m, *}{\widetilde{K}})$ for $z\in v^{\perp}$ and $\tau\in \mathbb{R}$. By \eqref{m-th-stein-1},  there exist 
\begin{align*}
\pmb{x}-v{}_{\bigcdot}\, s\in 
\Pi_{\phi, Q}^{m, *}{\widetilde{K}}\ \ \mathrm{and}\ \ \pmb{x}+v{}_{\bigcdot}\, t\in 
\Pi_{\phi, Q}^{m, *}{\widetilde{K}}
\end{align*}
with $\pmb{x}\in V^m(v)$, $s\in M_{1,m}(\mathbb{R})$ and $t\in M_{1,m}(\mathbb{R})$, such that
\begin{align}\label{6.2-2}
(z+\tau v){}_{\bigcdot}\, \iota=\pmb{x}+v{}_{\bigcdot}\, \frac{s+t}{2}.
\end{align}
Meanwhile, the following inclusion can be verified: 
\begin{align}\label{6.2-3}
(\Pi_{\phi, Q}^{m,*}{\widetilde{K}}){}_{\bigcdot}\, \kappa^{\mathrm{T}}
\subset
\Pi_{\psi}^{*}{\widetilde{K}}.
\end{align}
On the one hand, it can be checked that, by using the approach similar to the proof of \eqref{formula-6-1},
$$\Pi_{\phi, Q{}_{\bigcdot}\, \iota^{\mathrm{T}}{}_{\bigcdot}\, \kappa}^{m}{\widetilde{K}}=(\Pi_{\phi, Q}^{m}{\widetilde{K}}){}_{\bigcdot}\, \iota^{\mathrm{T}}{}_{\bigcdot}\, \kappa.$$
Combined with \eqref{formula-6-1}, one gets, for $P=\iota^{\mathrm{T}}{}_{\bigcdot}\, \kappa$, $$\Pi_{\phi, Q{}_{\bigcdot}\, P}^{m}{\widetilde{K}}=\Pi_{\phi, Q{}_{\bigcdot}\, \iota^{\mathrm{T}}{}_{\bigcdot}\, \kappa}^{m}{\widetilde{K}}=(\Pi_{\phi, Q}^{m}{\widetilde{K}}){}_{\bigcdot}\, \iota^{\mathrm{T}}{}_{\bigcdot}\, \kappa=(\Pi_{\psi}{\widetilde{K}}){}_{\bigcdot}\, \kappa.$$ 
On the other hand,   $\Pi_{\phi, Q{}_{\bigcdot}\, P}^{m}{\widetilde{K}}\subset \Pi_{\phi, Q}^{m}{\widetilde{K}}$ due to  $Q{}_{\bigcdot}\, P\subset Q$ and the fact that $\phi$ is strictly increasing on $[0,\infty)$. 
Thus,  
\begin{align*}
(\Pi_{\psi}{\widetilde{K}}){}_{\bigcdot}\, \kappa\subset \Pi_{\phi, Q}^{m}{\widetilde{K}}, \ \mathrm{and\ \ equivalently\ }\ \  \big(\Pi_{\phi, Q}^{m}{\widetilde{K}}\big)^*=\Pi_{\phi, Q}^{m,*}{\widetilde{K}}\subset \big((\Pi_{\psi}{\widetilde{K}}){}_{\bigcdot}\, \kappa\big)^*.\end{align*} 
That is, for any $\pmb{z}\in \big(\Pi_{\phi, Q}^{m}{\widetilde{K}}\big)^*, $ one has $\pmb{z}\in \big((\Pi_{\psi}{\widetilde{K}}){}_{\bigcdot}\, \kappa\big)^*.$ It follows from  \eqref{L-star} and $
\pmb{x}\bigcdot \pmb{y}=\mathrm{tr}(\pmb{x}^{\mathrm{T}}{}_{\bigcdot}\, \pmb{y})$ for $\pmb{x},\pmb{y}\!\in\!M_{n,m}(\mathbb{R})$ that $ \mathrm{tr}\big({\pmb{z}^{\mathrm{T}}{}_{\bigcdot}\, (x{}_{\bigcdot}\, \kappa)}\big)\leq 1$ for all $x\in\Pi_{\psi}{\widetilde{K}}.$  Note that $
\mathrm{tr}({\pmb{z}^{\mathrm{T}}{}_{\bigcdot}\, (x{}_{\bigcdot}\, \kappa)})=\mathrm{tr}(\kappa{}_{\bigcdot}\, {\pmb{z}^{\mathrm{T}}{}_{\bigcdot}\, x}).$ Thus, $\mathrm{tr}(\kappa{}_{\bigcdot}\, {\pmb{z}^{\mathrm{T}}{}_{\bigcdot}\, x})  \leq 1$ for all $x\in\Pi_{\psi}{\widetilde{K}}$.  It follows again from \eqref{L-star} that $(\kappa{}_{\bigcdot}\, {\pmb{z}^{\mathrm{T}}})^{\mathrm{T}}=\pmb{z}{}_{\bigcdot}\, \kappa^{\mathrm{T}}\in \Pi_{\psi}^*{\widetilde{K}}$, which implies \eqref{6.2-3}.  

As $\pmb{x}-v{}_{\bigcdot}\, s\in 
\Pi_{\phi, Q}^{m, *}{\widetilde{K}}$ and $\pmb{x}+v{}_{\bigcdot}\, t \in \Pi_{\phi, Q}^{m, *}{\widetilde{K}}$, it can be checked from \eqref{6.2-3} that  
\begin{align}\label{boundary-pi-psi-k}
 (\pmb{x}-v{}_{\bigcdot}\, s){}_{\bigcdot}\, \kappa^{\mathrm{T}}=\pmb{x}{}_{\bigcdot}\, \kappa^{\mathrm{T}}-v{}_{\bigcdot}\, s{}_{\bigcdot}\, \kappa^{\mathrm{T}}\in \Pi_{\psi}^{*}{\widetilde{K}}\ \ \mathrm{and}\ \ (\pmb{x}+v{}_{\bigcdot}\, t){}_{\bigcdot}\, \kappa^{\mathrm{T}}=\pmb{x}{}_{\bigcdot}\, \kappa^{\mathrm{T}}+v{}_{\bigcdot}\, t{}_{\bigcdot}\, \kappa^{\mathrm{T}}\in \Pi_{\psi}^{*}{\widetilde{K}}.
\end{align}
Here, $s{}_{\bigcdot}\, \kappa^{\mathrm{T}}$ and $t{}_{\bigcdot}\, \kappa^{\mathrm{T}}$ are constants. Observe that $z\in v^{\perp}$ and $\pmb{x}\in V^m(v)$. By left-multiplying both sides of \eqref{6.2-2} by $v^{\mathrm{T}}$, one can get 
\begin{align*}
\tau{}_{\bigcdot}\, \iota=\frac{s+t}{2} \ \mathrm{and}\ \ z{}_{\bigcdot}\, \iota=\pmb{x}.
\end{align*}
As $\iota{}_{\bigcdot}\, \kappa^{\mathrm{T}}=\mathrm{tr}(P)=1$,  one gets $$\tau=\frac{s{}_{\bigcdot}\, \kappa^{\mathrm{T}}+t{}_{\bigcdot}\, \kappa^{\mathrm{T}}}{2}\ \ \mathrm{and}\ \ z=\pmb{x}{}_{\bigcdot}\, \kappa^{\mathrm{T}}.$$  Therefore,  \eqref{equal-6.1-2.1} holds because $s{}_{\bigcdot}\, \kappa^{\mathrm{T}}\in \R$, $t{}_{\bigcdot}\, \kappa^{\mathrm{T}}\in \R$ and 
\begin{align*}
z+\tau v=\pmb{x}{}_{\bigcdot}\, \kappa^{\mathrm{T}}+\frac{(s{}_{\bigcdot}\, \kappa^{\mathrm{T}})v+(t{}_{\bigcdot}\, \kappa^{\mathrm{T}})v}{2}=\pmb{x}{}_{\bigcdot}\, \kappa^{\mathrm{T}}+ v{}_{\bigcdot}\,  \Big( \frac{s{}_{\bigcdot}\, \kappa^{\mathrm{T}}+ t{}_{\bigcdot}\, \kappa^{\mathrm{T}}}{2}\Big) \in S_v\Pi_{\psi}^{*}{\widetilde{K}},
\end{align*} where the last one follows from \eqref{stein-sym-1} and \eqref{boundary-pi-psi-k}. In particular, \eqref{equal-6.1-2} holds as desired.   

Let $\widetilde{K}\in \mathcal{K}_{(o)}^{n,1}$ be such that  \begin{align*}
\widetilde{K}=\Big\{z+\tau v\in M_{n,1}(\mathbb{R}): z\in P_{v^{\perp}}\widetilde{K}\  \mathrm{and}\ -\widetilde{f}_1(z)\leq \tau\leq \widetilde{f}_2(z)\Big\}.
\end{align*} 
As the origin $o$ is in the interior of $\widetilde{K}$, one sees that $\langle \widetilde{f}_1\rangle>0$ and $\langle \widetilde{f}_2\rangle>0$. Let   $\xi\in v^{\perp}$ and $s,t\in\mathbb{R}$ with $-s\ne t$ be such that \begin{align}\label{xi-vs-xi-tv}
\xi-sv\in\partial{\Pi_{\psi}^{ *}\widetilde{K}}\ \ \mathrm{and}\ \  \xi+tv\in\partial{\Pi_{\psi}^{*}\widetilde{K}}. 
\end{align}
 Then, $h_{\Pi_{\psi}\widetilde{K}}(\xi-sv)=1$ and $h_{\Pi_{\psi}\widetilde{K}}(\xi+tv)=1$. Employing \eqref{1.2} and \eqref{1.1} (by letting, for instance, $m=1$, $\Phi=\phi\circ h_{Q{}_{\bigcdot}\, \iota^{\mathrm{T}}}$, and $\pmb{\xi}=\xi$), one gets  
\begin{align} 
 n V_n(\widetilde{K})
&=\int_{P_{v^{\perp}}\widetilde{K}} \bigg[\phi\circ h_{Q{}_{\bigcdot}\, \iota^{\mathrm{T}}}\bigg(\frac{\!-\nabla \widetilde{f}_1(z){\bigcdot}{\xi}\!+\!s}{\langle \widetilde{f}_1\rangle(z)}\bigg)\langle \widetilde{f}_1\rangle(z) \!+\!\phi\circ h_{Q{}_{\bigcdot}\, \iota^{\mathrm{T}}}\bigg(\frac{\!-\nabla \widetilde{f}_2(z){\bigcdot}{\xi}\!-\!s}{\langle \widetilde{f}_2\rangle(z)}\bigg)\langle \widetilde{f}_2\rangle(z)\bigg] dz \label{6.1-3-1}  \\  
 &=\int_{P_{v^{\perp}}\widetilde{K}} \bigg[\phi\circ h_{Q{}_{\bigcdot}\, \iota^{\mathrm{T}}}\bigg(\frac{\!-\nabla \widetilde{f}_1(z){\bigcdot}{\xi}\!-\!t}{\langle \widetilde{f}_1\rangle(z)}\bigg)\langle \widetilde{f}_1\rangle(z)\!+\!\phi\circ h_{Q{}_{\bigcdot}\, \iota^{\mathrm{T}}}\bigg(\frac{\!-\nabla \widetilde{f}_2(z){\bigcdot}{\xi}\!+\!t}{\langle \widetilde{f}_2\rangle(z)}\bigg)\langle \widetilde{f}_2\rangle(z)\bigg]dz.\label{6.1-3-2} 
\end{align} 

It follows from \eqref{xi-vs-xi-tv} that $\xi+rv\in\partial{S_v\Pi_{\psi}^{ *}\widetilde{K}}$ with $r=\frac{t+s}{2}$. By \eqref{equal-6.1-2},  one has,  $\xi+rv\in \partial \Pi_{\psi}^{*}S_v\widetilde{K}$ and hence $h_{\Pi_{\psi}S_v\widetilde{K}}(\xi+rv)=1.$ Due to \eqref{1.2}, \eqref{1.1} and \eqref{def-orlicz-petty-body-psi}, one gets 
\begin{align}\label{6-4-1}
\nonumber
\!\!\!n V_n(S_v\widetilde{K})=&\int_{S^{n-1}}\!\!\!\! \psi\bigg(\!\frac{u \bigcdot  ({\xi}+rv)}{h_{S_v{\widetilde{K}}}(u)}\!\bigg)h_{S_v{\widetilde{K}}}(u)dS_{S_v\widetilde{K}}(u) \\
=&\!\int_{P_{v^{\perp}}{\widetilde{K}}}\!\!\!\! \phi\circ h_{Q{}_{\bigcdot}\, \iota^{\mathrm{T}}}\bigg(\!\frac{-\nabla(\widetilde{f}_1\!+\!\widetilde{f}_2)(z) \bigcdot {\xi}-(s+t)}{\langle \widetilde{f}_1+\widetilde{f}_2\rangle(z)}\!\bigg)\frac{\langle \widetilde{f}_1\!+\!\widetilde{f}_2\rangle(z)}{2} dz \nonumber \\
\nonumber
&\ \  +\!\int_{P_{v^{\perp}}{\widetilde{K}}}\!\!\!\! \phi\circ h_{Q{}_{\bigcdot}\, \iota^{\mathrm{T}}}\bigg(\frac{-\nabla(\widetilde{f}_1\!+\!\widetilde{f}_2)(z)\bigcdot {\xi}+(s+t)}{\langle \widetilde{f}_1\!+\! \widetilde{f}_2\rangle(z)}\bigg)\frac{\langle \widetilde{f}_1\!+\!\widetilde{f}_2\rangle(z)}{2}  dz\\
\leq &\!\int_{P_{v^{\perp}}{\widetilde{K}}}\!\!\!\! \phi \bigg(\!\frac{h_{Q{}_{\bigcdot}\, \iota^{\mathrm{T}}}\big(-\!\nabla \widetilde{f}_1(z) \bigcdot {\xi}\!-\!t\big)\!+\!h_{Q{}_{\bigcdot}\, \iota^{\mathrm{T}}}\big(-\nabla \widetilde{f}_2(z) \bigcdot {\xi}\!-\!s\big)}{\langle \widetilde{f}_1+\widetilde{f}_2\rangle(z)}\!\bigg)\frac{\langle \widetilde{f}_1\!+\!\widetilde{f}_2\rangle(z)}{2} dz \nonumber \\
 &\ \  +\!\int_{P_{v^{\perp}}{\widetilde{K}}}\!\!\!\! \phi \bigg(\!\frac{h_{Q{}_{\bigcdot}\, \iota^{\mathrm{T}}}\big(-\!\nabla \widetilde{f}_1(z) \bigcdot {\xi}\!+\!s\big)\!+\!h_{Q{}_{\bigcdot}\, \iota^{\mathrm{T}}}\big(-\nabla \widetilde{f}_2(z)\bigcdot{\xi}\!+\!t\big)}{\langle \widetilde{f}_1+\widetilde{f}_2\rangle(z)}\!\bigg)\frac{\langle \widetilde{f}_1\!+\!\widetilde{f}_2\rangle(z)}{2} dz, 
\end{align}
where the last inequality follows from the convexity of $h_{Q{}_{\bigcdot}\, \iota^{\mathrm{T}}}$ and the fact that $\phi$ is strictly increasing on $[0,\infty)$. This, together with \eqref{6.1-3-1}, \eqref{6.1-3-2} and the convexity of $\phi$, yields that
\begin{align}\label{6-4-2}
\nonumber 
\!\!n V_n(S_v\widetilde{K})
\leq&\ \frac{1}{2}\int_{{P_{v^{\perp}}\widetilde{K}} }\phi\circ h_{Q{}_{\bigcdot}\, \iota^{\mathrm{T}}}\bigg(\frac{-\nabla \widetilde{f}_1(z)\bigcdot{\xi}-t}{\langle \widetilde{f}_1\rangle(z)}\bigg)\langle \widetilde{f}_1\rangle(z) dz\\
\nonumber
&\!+\!\frac{1}{2}\int_{{P_{v^{\perp}}\widetilde{K}} }\phi\circ h_{Q{}_{\bigcdot}\, \iota^{\mathrm{T}}}\bigg(\frac{-\nabla \widetilde{f}_2(z)\bigcdot{\xi}-s}{\langle \widetilde{f}_2\rangle(z)}\bigg)\langle \widetilde{f}_2\rangle(z) dz\\
\nonumber
&\!+\!\frac{1}{2}\int_{{P_{v^{\perp}}\widetilde{K}}}\phi\circ h_{Q{}_{\bigcdot}\, \iota^{\mathrm{T}}}\bigg(\frac{-\nabla \widetilde{f}_1(z)\bigcdot{\xi}+s}{\langle \widetilde{f}_1\rangle(z)}\bigg)\langle \widetilde{f}_1\rangle(z)dz\\
&\!+\!\frac{1}{2}\int_{{P_{v^{\perp}}\widetilde{K}}}\phi\circ h_{Q{}_{\bigcdot}\, \iota^{\mathrm{T}}}\bigg(\frac{-\nabla \widetilde{f}_2(z)\bigcdot{\xi}+t}{\langle \widetilde{f}_2\rangle(z)}\bigg)\langle \widetilde{f}_2\rangle(z)dz=n V_n(\widetilde{K}).
\end{align}

As $ V_n(S_v\widetilde{K})= V_n(\widetilde{K})$, equalities hold in both \eqref{6-4-1} and \eqref{6-4-2}. As $\phi$ is strictly convex, one gets,  for almost all $z\in P_{v^{\perp}}{\widetilde{K}}$, \begin{align}
&h_{Q{}_{\bigcdot}\, \iota^{\mathrm{T}}}\bigg(\frac{-\nabla \widetilde{f}_1(z)\bigcdot{\xi}-t}{\langle \widetilde{f}_1\rangle(z)}\bigg)=h_{Q{}_{\bigcdot}\, \iota^{\mathrm{T}}}\bigg(\frac{-\nabla \widetilde{f}_2(z)\bigcdot{\xi}-s}{\langle \widetilde{f}_2\rangle(z)}\bigg), \nonumber\\
&h_{Q{}_{\bigcdot}\, \iota^{\mathrm{T}}}\bigg(\frac{-\nabla \widetilde{f}_1(z)\bigcdot{\xi}+s}{\langle \widetilde{f}_1\rangle(z)}\bigg)= h_{Q{}_{\bigcdot}\, \iota^{\mathrm{T}}}\bigg(\frac{-\nabla \widetilde{f}_2(z)\bigcdot{\xi}+t}{\langle \widetilde{f}_2\rangle(z)}\bigg).  \label{6.18-2}
\end{align}
Since $o\in\mathrm{int}\Pi_{\psi}^*\widetilde{K}$, we can let $\xi=o\in \mathrm{int}P_{v^\perp}K$. Then, there exist $s_0=s_0(v)>0$ and $t_0=t_0(v)>0$ such that $-s_0v\in\partial \Pi_{\psi}^*\widetilde{K}$ and $t_0v\in\partial \Pi_{\psi}^*\widetilde{K}$. By \eqref{6.18-2}, one gets
\begin{align*}
h_{Q{}_{\bigcdot}\, \iota^{\mathrm{T}}}\bigg(\frac{s_0}{\langle \widetilde{f}_1\rangle(z)}\bigg)= h_{Q{}_{\bigcdot}\, \iota^{\mathrm{T}}}\bigg(\frac{t_0}{\langle \widetilde{f}_2\rangle(z)}\bigg). 
\end{align*}
It follows from \eqref{hom-1-1}, \eqref{positive-def-p-q},   $s_0>0$ and $t_0>0$ that
\begin{align*}
\frac{s_0h_Q(\iota)}{\langle \widetilde{f}_1\rangle(z)}=\frac{t_0h_Q(\iota)}{\langle \widetilde{f}_2\rangle(z)}.
\end{align*}
Since $h_Q(\iota)>0$, $\langle \widetilde{f}_1\rangle>0$  and $\langle \widetilde{f}_2\rangle>0$, the above equality implies that
 \begin{align}\label{langle-f2-cf1}
 \langle \widetilde{f}_2\rangle(z)=\frac{t_0}{s_0}\langle \widetilde{f}_1\rangle(z), \ \ \mathrm{i.e.,}\ \ \Big(\widetilde{f}_2-\frac{t_0}{s_0}\widetilde{f}_1\Big)(z)=\nabla\Big(\widetilde{f}_2-\frac{t_0}{s_0} \widetilde{f}_1\Big)(z)\bigcdot z.
 \end{align}

We now show that $t_0(v)=s_0(v)$ for any $v\in \sphere$.  
As $P_{v^{\perp}}\widetilde{K}\subset M_{n,1}(\mathbb{R})$ is an $(n-1)$-dimensional convex body, its relative interior, $\mathrm{relint}P_{v^{\perp}}\widetilde{K}$, can be obtained by the union of an increasing sequence of $(n-1)$-dimensional convex bodies $\left\{E_k\right\}_{k=1}^{\infty}$, namely, 
\begin{align}\label{E1-E2-Einfty}
\mathrm{relint}P_{v^{\perp}}\widetilde{K}=\bigcup_{k=1}^{\infty} E_k \ \ \mathrm{where}\ \  E_1 \subset E_2 \subset \cdots \subset\mathrm{relint}P_{v^{\perp}}\widetilde{K}.
\end{align}
For each $j=1,2,$  $-\widetilde{f}_j$ is convex, and by \cite[Theorem 24.7]{Rockafellar}, the subdifferential of $-\widetilde{f}_j$ on each $E_k$ is bounded for  $k\in\mathbb{N}$. Note that   $-\widetilde{f}_j$ is differentiable $\mathcal{H}^{n-1}$-almost everywhere on $E_k$  for  $k\in\mathbb{N}$. 

Let $k\in \mathbb{N}$ be fixed. At each point $z\in E_k$ where $-\widetilde{f}_j$ is differentiable,   $\nabla (-\widetilde{f}_j)(z)=-\nabla (\widetilde{f}_j)(z)$ is the unique subgradient of $-\widetilde{f}_j$ at $z$ (see \cite[Theorem 25.1]{Rockafellar}). Together with the boundedness of subdifferential (see \cite[Theorem 24.7]{Rockafellar}), it follows that a constant $c_k>0$ can be found such that 
\begin{align} \label{max-nabla-1}
\max_{j=1, 2} \big|\nabla \widetilde{f}_j(z)\big| = \max_{j=1, 2} \big|\nabla (-\widetilde{f}_j)(z)\big| <c_k \ \ \mathrm{for\ almost\ all}\ z\in E_k.   
\end{align} 
As $o\in\mathrm{int}\Pi_{\psi}^*\widetilde{K}$ and $o\in \mathrm{int} \widetilde{K}$, there exists a constant $r_0>0$ such that $r_0B_2^n\subset \Pi_{\psi}^*\widetilde{K}$ and $r_0B_2^n\subset \widetilde{K}$. Let $\xi\in P_{v^{\perp}}\widetilde{K} $ be such that $|\xi|=\min\{\frac{r_0}{2 c_k}, \frac{r_0}{2}\}$. It can be checked that there exist $s(\xi)>0$ and $t(\xi)>0$ such that 
 $$\xi-s(\xi)v\in \partial \Pi_{\psi}^*\widetilde{K}\ \mathrm{and}\ \ \xi+t(\xi)v\in \partial \Pi_{\psi}^*\widetilde{K}.$$ Clearly, $|\xi|^2+s(\xi)^2\geq r_0^2$ as $r_0B_2^n\subset \Pi_{\psi}^*\widetilde{K}$. By \eqref{max-nabla-1}, $|\xi|=\min\{\frac{r_0}{2 c_k}, \frac{r_0}{2}\}$, and the Cauchy-Schwarz inequality, one gets, for almost all $z\in E_k$,  $$s(\xi)-\nabla \widetilde{f}_1(z)\bigcdot{\xi}\geq \sqrt{r_0^2-|\xi|^2}-c_k|\xi|\geq \sqrt{r_0^2-\frac{r_0^2}{4}}-\frac{c_k  r_0}{2 {c_k}}=\frac{(\sqrt{3}-1)r_0}{2}>0.$$ Similarly,  $-\nabla \widetilde{f}_2(z)\bigcdot{\xi}+t(\xi)>0$ for almost all $z\in E_k$. Together with \eqref{hom-1-1}, \eqref{positive-def-p-q}, \eqref{6.18-2} and the fact that $h_Q(\iota)>0$, one can get 
\begin{align*}
\frac{-\nabla \widetilde{f}_1(z)\bigcdot{\xi}+s(\xi)}{\langle \widetilde{f}_1\rangle(z)} =\frac{ -\nabla \widetilde{f}_2(z)\bigcdot{\xi}+t(\xi)}{\langle \widetilde{f}_2\rangle(z)}. 
\end{align*} By \eqref{langle-f2-cf1}, the following holds: 
\begin{align*} 
\frac{t_0}{s_0}\Big(\!\!-\!\!\nabla \widetilde{f}_1(z)\bigcdot{\xi}+s(\xi)\Big)= -\nabla \widetilde{f}_2(z)\bigcdot{\xi}+t(\xi).  
\end{align*} After rearrangement, one has 
\begin{align}\label{1-frac-t-s-xi}
 \Big(\nabla\widetilde{f}_2-\frac{t_0}{s_0}\nabla\widetilde{f}_1\Big)(z)\bigcdot{\xi}= \nabla\Big(\widetilde{f}_2-\frac{t_0}{s_0}\widetilde{f}_1\Big)(z)\bigcdot{\xi}=t(\xi)-\frac{t_0}{s_0}s(\xi).
\end{align}

Note that $z$ and $\xi$ are independent of each other. Thus, for almost all $z\in E_k$, \eqref{1-frac-t-s-xi} asserts that $\nabla\big(\widetilde{f}_2-\frac{t_0}{s_0}\widetilde{f}_1\big)(z)$ lies in a hyperplane with normal direction $\frac{\xi}{|\xi|}$. As $\frac{\xi}{|\xi|}\in S^{n-1}\cap v^{\perp}$ is arbitrary, for each fixed $k$, one can indeed get $\nabla\big(\widetilde{f}_2-\frac{t_0}{s_0}\widetilde{f}_1\big)(z)=\widetilde{w}_k$ for almost all $z\in E_k$, where $\widetilde{w}_k\in M_{n,1}(\mathbb{R})$ is independent of $z$. It follows from \eqref{langle-f2-cf1} that, for each $k\in\mathbb{N}$,
\begin{align}\label{widetilde-wk}
\Big(\widetilde{f}_2-\frac{t_0}{s_0}\widetilde{f}_1\Big)(z)=\widetilde{w}_k \bigcdot z \ \ \mathrm{for\ almost\ all}\ \ z\in E_k.
\end{align} 
Note that $f_j$, $j=1,2,$ are continuous on each $E_k\subset \mathrm{relint}P_{v^{\perp}}\widetilde{K}$. Thus, the above equation holds for all $z\in E_k$. Indeed,  as $E_1 \subset E_2 \subset \cdots \subset\mathrm{relint}P_{v^{\perp}}\widetilde{K}$ (see \eqref{E1-E2-Einfty}), $\widetilde{w}_k$ turns out to be the same for all $k\in\mathbb{N}$. For convenience, such a constant is denoted by $\widetilde{w}$, that is,  $\widetilde{w}=w_k$ for $k\in\mathbb{N}$. It follows from  \eqref{widetilde-wk} that, for $z\in\mathrm{relint}P_{v^{\perp}}\widetilde{K}$,  
\begin{align}\label{f-2-t0s0-f1=wz}
\Big(\widetilde{f}_2-\frac{t_0}{s_0}\widetilde{f}_1\Big)(z)=\widetilde{w} \bigcdot z.
\end{align} Again, \eqref{f-2-t0s0-f1=wz} holds for all $z\in P_{v^{\perp}}\widetilde{K}$.   
Thus, $\widetilde{f}_2-\frac{t_0}{s_0} \widetilde{f}_1$ is a linear function on $P_{v^{\perp}}\widetilde{K}$. Let $$r=r(v)=\frac{t_0}{s_0}:\Big(1+\frac{t_0}{s_0}\Big).$$ By the linearity of $\widetilde{f}_2-\frac{t_0}{s_0} \widetilde{f}_1$, it follows that the points, which  divide the (directed) chords of $K$ parallel to $v$ in the proportion $r:1-r$, lie in a hyperplane. As $v\in S^{n-1}$ is arbitrary, Theorem \ref{Gruber-theorem-3} implies that $K$ is an ellipsoid. In particular,  $r=r(v)=\frac{1}{2}$ and hence $t_0=s_0$ (i.e., $t_0(v)=s_0(v)$) holds for all $v\in \sphere$. Together with \eqref{f-2-t0s0-f1=wz}, one gets
\begin{align*}
\widetilde{f}_1(z)-\widetilde{f}_2 (z)={\widetilde{w}}\bigcdot z.
\end{align*}
Thus, the midpoints of the chords of $\widetilde{K}$ parallel to $v$ lie in the subspace:
$$\Big\{z+\frac{{\widetilde{w}}\bigcdot z}{2} v: z\in P_{v^{\perp}}\widetilde{K}\Big\}.$$ Therefore, $\widetilde{K}$ is an origin-symmetric ellipsoid (see e.g., \cite[Theorem 10.2.1]{schneider}) as desired. 
\end{proof}

\vskip 2mm \noindent  {\bf Acknowledgement.}  We are in great debt to the referee for many wonderful and valuable comments and suggestions, which greatly improve the presentation of the paper. The
research of DY was supported by a NSERC grant, Canada.
 The research of ZZ was supported by NSFC (No. 12301071), Natural Science Foundation of
Chongqing, China CSTC (No. CSTB2024NSCQ-MSX1085) and the Science and Technology Research Program of Chongqing Municipal
Education Commission (No. KJQN202201339).

\vspace{16pt}

\noindent Xia Zhou, Department of Mathematics and Statistics, Memorial University of Newfoundland, St. John’s, Newfoundland, A1C 5S7, Canada\\
\textit{Email address}: xiaz@mun.ca
\vspace{8pt}

\noindent Deping Ye, Department of Mathematics and Statistics, Memorial University of Newfoundland, St. John’s, Newfoundland, A1C 5S7, Canada\\
\textit{Email address}: deping.ye@mun.ca
\vspace{8pt}

\noindent Zengle Zhang, Key Laboratory of Group and Graph Theories and Applications, Chongqing University of Arts and Sciences, Yongchuan, Chongqing, 402160, China\\
\textit{Email address}: 
zhangzengle128@163.com

\end{document}